\documentclass{article}

\usepackage[english]{babel}

\usepackage[letterpaper,top=2cm,bottom=2cm,left=3cm,right=3cm,marginparwidth=1.75cm]{geometry}

\usepackage{amsmath}
\usepackage{graphicx}
\usepackage[colorlinks=true, allcolors=blue]{hyperref}

\usepackage{indentfirst}
\usepackage{amsfonts}
\usepackage{mathrsfs}
\usepackage{float}
\usepackage{underscore}
\usepackage[ruled,vlined]{algorithm2e}
\usepackage{algpseudocode}
\usepackage{algcompatible}
\usepackage{mathtools}
\usepackage{amsthm}
\usepackage{amssymb}
\newtheorem{theorem}{Theorem}[section]
\newtheorem{corollary}{Corollary}[theorem]
\newtheorem{lemma}[theorem]{Lemma}
\newtheorem{assumption}{Assumption}
\newtheorem{remark}{Remark}[section]
\theoremstyle{definition}

\title{A Gradient Complexity Analysis for Minimizing the Sum of Strongly Convex Functions with Varying Condition Numbers}

\author{Nuozhou Wang\thanks{Department of Industrial and System Engineering, University of Minnesota, wang9886@umn.edu},
\and
Shuzhong Zhang\thanks{Department of Industrial and System Engineering, University of Minnesota, zhangs@umn.edu},
}

\begin{document}

\maketitle

\begin{abstract}
A popular approach to minimize a finite-sum of convex functions is {\it stochastic gradient descent}\/ (SGD) and its variants. Fundamental research questions associated with SGD include: (i) To find a lower bound on the number of times that the gradient oracle of each individual function must be assessed in order to find an $\epsilon$-minimizer of the overall objective; (ii) To design algorithms which guarantee to find an $\epsilon$-minimizer of the overall objective in expectation at no more than a certain number of times (in terms of $1/\epsilon$) that the gradient oracle of each functions needs to be assessed (i.e., upper bound). If these two bounds are at the same order of magnitude, then the algorithms may be called {\it optimal}. Most existing results along this line of research typically assume that the functions in the objective share the same condition number. In this paper, the first model we study is the problem of minimizing the sum of finitely many strongly convex functions whose condition numbers are all different. We propose an SGD method for this model, and show that it is optimal in gradient computations, up to a logarithmic factor. We then consider a constrained separate block optimization model, and present lower and upper bounds for its gradient computation complexity. Next, we propose to solve the Fenchel dual of the constrained block optimization model via the SGD we introduced earlier, and show that it yields a lower iteration complexity than solving the original model by the ADMM-type approach.
Finally, we extend the analysis to the general composite convex optimization model, and obtain gradient-computation complexity results under certain conditions.
\end{abstract}


\vspace{0.5cm}

\vspace{0.5cm}

\section{Introduction}
Let us first consider the following model
\begin{equation}
    \min_{x\in \Gamma} F(x)=\sum_{i=1}^{m} g_i(x),  \label{pro}
\end{equation}
where $g_i(x)$'s are smooth and convex functions, and $\Gamma$ is a closed convex set. When $m$ is large, computing the exact gradient of $\sum_{i=1}^{m} g_i(x)$ becomes costly, and it gives rise to the so-called stochastic gradient descent (SGD) method, which only computes a single gradient at random or a mini-batch of randomly selected gradients at each iteration. However, SGD converges with a sublinear rate of $1/\epsilon$ even if all the functions $g_i(\cdot)$'s are strongly convex and smooth.
The variance reduction technique was therefore introduced to improve the convergence of SGD to be linear, such as stochastic average gradient (SAG)~\cite{schmidt2017minimizing}, SAGA~\cite{defazio2014saga} and stochastic variance reduced gradient (SVRG)~\cite{johnson2013accelerating}, among others.
In this paper, we shall use the terminology of {\it gradient complexity}\/ to refer to the amount of times that the gradients
of functions $g_i$'s need be computed before reaching an $\epsilon$ optimal solution.
In particular, the gradient
complexity of variance-reduced SGD for solving \eqref{pro} can be shown to be $O((m+L/\mu)\log(1/\epsilon))$ if we assume all $g_i(\cdot)$'s to be $\mu$ strongly convex and $L$-smooth.
Stochastic variance-reduced gradient methods can also be accelerated, in the spirit of Nesterov's acceleration. Katyusha \cite{allen2017katyusha} and Loopless Katyusha \cite{kovalev2020don} are accelerated variants of the SVRG, and SAGA with Sampled Negative Momentum (SSNM) \cite{zhou2019direct} is an accelerated variant of SAGA, and they 
have reached an improved gradient complexity of $O((m+\sqrt{mL/\mu})\log(1/\epsilon))$.

On the side of informational lower bounds on the iteration complexity for the gradient methods,
\cite{nemirovsky1992information} gives a lower bound $\Omega(1/\sqrt{\epsilon})$ on the iteration complexity for convex minimization; \cite{nesterov2003introductory} and \cite{nesterov2018lectures} present a quadratic function showing the oracle complexity of strongly convex function is $\Omega(\sqrt{L/\mu}\log(1/\epsilon))$, which matches the upper bound, hence is optimal; \cite{zhang2019lower} extends this result to the strongly convex-concave minimax model, leading to a lower bound of $\Omega\left(\sqrt{\frac{L_{x}}{\mu_{x}}+\frac{L_{x y}^{2}}{\mu_{x} \mu_{y}}+\frac{L_{y}}{\mu_{y}}} \log \left(\frac{1}{\epsilon}\right)\right)$.
For the finite-sum problem \eqref{pro}, a lower bound on the gradient computations has been established in~\cite{lan2018optimal} to be $\Omega((m+\sqrt{mL/\mu})\log(1/\epsilon))$, which matches the required gradient complexity of Katyusha and/or SSNM, hence optimal.

Most of the above-discussed upper and lower bounds assume the functions $g_i(\cdot)$'s to share the same Lipschitz gradient constant. However, the Lipschitz gradient constants for $g_i(\cdot)$'s may actually vary significantly. For example, weighted logistic regression is commonly used for imbalanced and rare events data \cite{king2001logistic}. In this paper, we consider the case where $g_i$'s may have different gradient Lipschitz constant $L_i>0$, and different strong convexity parameter $\mu_i>0$. In this setting, we generalize the SSNM algorithm~\cite{zhou2019direct} and the lower bound in~\cite{lan2018optimal}, which turn out to both equal to $O\left(\left(m+\sum_{i=1}^{m}\sqrt{L_i} \bigg/\sqrt{\sum_{i=1}^m \mu_i}\right)\log(1/\epsilon)\right)$, hence is optimal. Note that~\cite{lan2018optimal} discusses the situation with different Lipschitz constant, and the upper and lower bounds match only when the summation of gradient Lipschitz constants of $g_i(\cdot)$'s is considered as a parameter.

This paper is organized as follows.
In Section~\ref{LB}, we present an adversarial example to give a lower bound on the gradient-computation complexity of the finite-sum problem. Section~\ref{UB} introduces the generalized SSNM to solve the finite-sum problem. In Section \ref{sec4}, we discuss the Fenchel duality for the finite-sum problem and present lower and upper bounds of the 
gradient-computation complexity. In Section \ref{sec5}, a generalized SSNM is introduced to solve the multi-block convex minimization problem via Fenchel duality. Section \ref{sec6} presents bounds on the gradient complexity of general composite optimization models. In Section \ref{sec7} we present some numerical experiments of the proposed methods. Finally, Section~\ref{sec8} concludes the paper.

\section{A Lower Bound on the Gradient Complexity} \label{LB} 

Recall our model~\eqref{pro}. 
Note that
a function $g$ is called $\mu$-strongly convex if
\begin{equation}
    g(y)\geq g(x)+\left\langle\nabla g(x),y-x\right\rangle+\frac{\mu}{2}\|y-x\|^2, \quad \forall x,y\in \mathcal{X}
\end{equation}
and
a function $g$ is called $L$-Lipschitz smooth if
\begin{equation}
    \|\nabla g(x)-\nabla g(y)\|\leq L\|x-y\|, \quad \forall x,y\in \mathcal{X} .
\end{equation}
Below is our basic assumption on~\eqref{pro}:
\begin{assumption}
In Model~\eqref{pro}, $g_i(x)$ is assumed to be $L_i$-smooth and $\mu_i$-strongly convex, $i=1,2,...,m$.
\end{assumption}

Different from the finite sum convex optimization settings that we discussed before, the gradient Lipschitz constant for each function $g_i(x)$ is now assumed to be different, $i=1,2,...,m$.

In order to prove the algorithm lower bound, we first review the quadratic functions adversarial example as in~\cite{nesterov2018lectures}. It is a chain-like function that has the so-called zero-chain property, meaning that any first-order method can only correctly identify coordinates of the solution one at a time. To simplify the analysis, the function is defined in $\ell_{2}$ as
\begin{equation}
    g(x)=\frac{L-\mu}{8}(x^{\top} A x-2x_{1})+\frac{\mu}{2}\|x\|^2,
\end{equation}
where
\begin{equation} \label{matrix-A}
A=\left(\begin{array}{rrrrr}
2 & -1 & 0  & 0  & \cdots \\
-1 & 2 & -1 & 0  & \cdots \\
0 & -1 & 2  & \ddots & \ddots \\
0 & 0 & \ddots & \ddots & \ddots \\
\vdots & \ddots & \ddots & \ddots & \vdots \\
\end{array}\right).
\end{equation}

We can see that $0 \preceq A \preceq 4 I$. So, $g(x)$ is $\mu$-strongly convex and $L$-smooth. We refer to \cite{nesterov2018lectures} for the following two lemmas.
\begin{lemma}\label{l1}
It has an optimal solution $(q,q^2,\cdots)$, where $q=\frac{\sqrt{\kappa}-1}{\sqrt{\kappa}+1}$ with $\kappa=L/\mu$.
\end{lemma}

\begin{lemma}\label{l2} For any solution $x^k$, if only the first $k$ entries can be non-zero, then
 \begin{eqnarray*}
    \|x^k-x^{\star}\|^2 &\geq& \left(\frac{\sqrt{\kappa}-1}{\sqrt{\kappa}+1}\right)^{2k}\|x^0-x^{\star}\|^2, \\
    g(x^k)-g(x^{\star}) &\geq& \frac{\mu}{2} \left(\frac{\sqrt{\kappa}-1}{\sqrt{\kappa}+1}\right)^{2k}\|x^0-x^{\star}\|^2.
\end{eqnarray*}
\end{lemma}
Since $\nabla^2 g(x)$ is a tri-diagonal operator, starting from the origin, after the $k$-th step of the first order method, only the first $k$ entries can become non-zero. So, the convergence rate is at least $q$ and the iteration complexity is at least $O(\sqrt{\kappa}\log(1/\epsilon))$.

For this problem, we refer to Nesterov's adversarial example and use multi-block functions~\cite{lan2018optimal}. Consider $x=(x_1,x_2,\cdots,x_m)$, where $x_i=\{x_{ij}\}_{j=1}^{\infty}\in \ell_{2}$. Let
\begin{equation}
    g_i(x)=\frac{L_i-\mu_i}{8}(x_i^{\top} A x_i-2\gamma_i x_{i1})+\frac{\mu_i}{2}\|x\|^2,
\end{equation}
where $\gamma_i$ is a constant and will be defined later. Then, $g_i$ is $L_i$-smooth and $\mu_i$-strongly convex.  Let $\mu=\sum_{i=1}^m \mu_i$. We have,
\[
F(x)= \sum_{i=1}^{m} g_i(x)=\sum_{i=1}^{m}\left(\frac{L_i-\mu_i}{8}(x_i^{\top} A x_i-2\gamma_i x_{i1})+\frac{\mu}{2}\|x_i\|^2\right).
\]
Observe that $F(x)$ is the sum of $m$ separate functions. By defining $G_i(x_i)=\frac{L_i-\mu_i}{8}(x_i^{\top} A x_i-2\gamma_i x_{i1})+\frac{\mu}{2}\|x_i\|^2$, we have $F(x)= \sum_{i=1}^{m}G_i(x_i)$.

Define $\nu_i=(L_i-\mu_i)/\mu+1$ and $q_i=\frac{\sqrt{\nu_i}-1}{\sqrt{\nu_i}+1}$. By applying Lemma \ref{l1} on function $G_i(x_i)$, we have the following lemma.
\begin{lemma}
The sequence $x^{\star}_{ij}=\gamma_i q_i^j$ ($j=1,2,...$) is the optimal solution for minimizing $G_i(x_i)$, where $i=1,2,...,m$.
\end{lemma}

Since $F(x)=\sum_{i=1}^{m}G_i(x_i)$ is separable, one observes that $(x^{\star}_1,x^{\star}_2,\cdots,x^{\star}_m)$ is the optimal solution for minimizing $F(x)$. Now, we consider an algorithm that starts from the initial point which is the origin. Letting $\gamma_i=\frac{\sqrt{1-q_i^2}}{q_i}$, we have $\|x_i^0-x_i^\star\|=1$. Let $K_i$ be the query times of $\nabla g_i$. Then, $K=K_1+K_2+\cdots+K_m$.
Note that only the first $K_i$ coordinates of $x_i^K$ can be nonzero. By applying Lemma \ref{l2} to the function $G_i(x_i)$, we have:
\begin{lemma}
\begin{equation}
    \|x_i^K-x^{\star}_i\|^2\geq \left(\frac{\sqrt{\nu_i}-1}{\sqrt{\nu_i}+1}\right)^{2K_i}\|x_i^0-x_i^{\star}\|^2,
\end{equation}
\begin{equation}
    G_i(x_i^K)-G_i(x^{\star}_i)\geq \frac{\mu}{2} \left(\frac{\sqrt{\nu_i}-1}{\sqrt{\nu_i}+1}\right)^{2K_i}\|x_i^0-x_i^{\star}\|^2,\,\, i=1,2,...,m.
\end{equation}
\end{lemma}

Therefore, 
\begin{align*}
    \epsilon>F(x^K)-F(x^{\star})&=\sum_{i=1}^{m} (G_i(x_i^{K})-G_i(x_i^{\star}))\\
    &\geq G_j(x_j^{K})-G_j(x_j^{\star})\\
    &\geq \frac{\mu}{2} \left(\frac{\sqrt{\nu_j}-1}{\sqrt{\nu_j}+1}\right)^{2K_j}\|x_j^0-x_j^{\star}\|^2\\
    &=\frac{\mu}{2} \left(\frac{\sqrt{\nu_j}-1}{\sqrt{\nu_j}+1}\right)^{2K_j}
\end{align*}
where the last step is because $\|x_j^0-x_j^{\star}\|^2=1$.
This gives us a lower bound on $K_j$:
\[
K_j \geq \frac 1 2 \frac{1}{\log \left(\frac{\sqrt{\nu_j}+1}{\sqrt{\nu_j}-1}\right)}\log\left(\frac{\mu}{2\epsilon}\right)
\geq
\frac 1 4 \left(\sqrt{\nu_j}-1\right) \log \left(\frac{\mu}{2\epsilon}\right), \,\,\,
j=1,2,...,m,
\]
assuming $\frac{\mu}{2\epsilon}>1$.
For notional simplicity, we assume: 
\begin{assumption} \label{assumption-cond}
Assume that $L_i-\mu_i>\mu/m$, where $i=1,2,...,m.$
\end{assumption}

In this case, $\nu_i= (L_i-\mu_i) /\mu+1>1/m+1$, $i=1,2,...,m$, and so assuming $\frac{\mu}{2\epsilon}>1$,
\begin{eqnarray*}
K_j &\geq& \frac 1 2 \frac{1}{\log \left(\frac{\sqrt{\nu_j}+1}{\sqrt{\nu_j}-1}\right)}\log\left(\frac{\mu}{2\epsilon}\right) \\
&\geq&
\frac 1 2 \frac{1}{\log \left(\frac{\sqrt{1/m+1}+1}{\sqrt{1/m+1}-1}\right)}\log\left(\frac{\mu}{2\epsilon}\right)\\
&\geq& \frac{1}{2\log(9m)}\log \left(\frac{\mu}{2\epsilon}\right),\, j=1,2,...,m.
\end{eqnarray*}
Thus, we have arrived at the following result.

\begin{theorem} \label{lower-iter-bound}
Under Assumption~\ref{assumption-cond}, by using any first-order method to get an $\epsilon$ optimal solution, the number of gradient computations is in general at least
\begin{eqnarray*}
    K &\geq& \Omega\left(\left(\sum_{i=1}^{m}(\sqrt{\nu_i}-1)+\sum_{i=1}^{m}\frac{1}{\log m}\right) \log\left(\frac{\mu}{\epsilon}\right) \right) \\
    &=& \Omega\left(\left(\frac{\sum_{i=1}^{m} \sqrt{L_i}}{\sqrt{\sum_{i=1}^m \mu_i }}  + \frac{m}{\log m} \right) \log\left(\frac{\sum_{i=1}^m \mu_i }{\epsilon}\right) \right) .
\end{eqnarray*}
\end{theorem}

\begin{remark}
The lower bound $\Omega((m+\sqrt{m \bar L/\bar \mu})\log(1/\epsilon))$ given by \cite{lan2018optimal} is a special case of the above result when $L_1=L_2=\cdots=L_m=\bar L$,  $\mu_1=\mu_2=\cdots=\mu_m = \bar \mu$, and $\bar L/\bar \mu =1+ \Omega (1)$, up to a logarithmic factor.
\end{remark}

\section{An Upper Bound on the Gradient Complexity} \label{UB}  
SSNM \cite{zhou2019direct} is an accelerated variant of SAGA to solve the finite sum problem. It applies the so-called negative momentum to SAGA, which then matches the lower bound on the gradient oracle in the case all $L_i$'s are the same. In this section, we aim to generalize SSNM to the setting where the Lipschitz constants $L_i$'s vary, and we shall show that the lower bound meets the upper bound in this case as well.

Recall the problem
\[
F(x)=\sum_{i=1}^m g_i(x)=\sum_{i=1}^m \hat{g}_i(x)+h(x),
\]
where we let $\hat{g}_i(x) := g_i(x)-\frac{\mu_i}{2}\|x\|^2$, $h(x) := \frac{\mu}{2}\|x\|^2$, where $\mu=\sum_{i=1}^m \mu_i$.
We use the following estimation of the gradient of the first term,
\[
\widetilde{\nabla}^{k}=\frac{1}{\pi_{i_{k}}}\left(\nabla \hat{g}_{i_k}(y_{i_k}^k)-\nabla \hat{g}_{i_k}(\phi_{i_k}^k)\right)+  \sum_{i=1}^m\nabla \hat{g}_{i}(\phi_{i}^k){\color{red},}
\]
where $i_k$ is sampled from $\{1,2,\cdots,m\}$ with the probability
\[
\pi_i=\frac{\sqrt{L_i}}{2\sum_{j=1}^m\sqrt{L_j}}+\frac{1}{2m},\,\, i=1,2,...,m.
\]

Observe the gradient estimator $\widetilde{\nabla}^{k}$ is biased with expectation $\sum_{i=1}^m\nabla \hat{g}_{i}(y_{i}^k)$, if we extend the definition to all $y_{i}^{k}:=\tau_{i} x^{k}+(1-\tau_{i}) \phi_{i}^{k}$.

The idea of setting the probability to randomly pick a function is inspired by \cite{lan2018optimal}. Our adversarial example shows that the number of $g_i(x)$ samples should be approximately proportional to $\sqrt{L_i}$ to get an $\epsilon$-solution. That explains the first part of the probability; the second part guarantees that each function will have at least $\frac{1}{2m}$ probability to be chosen, which is important in keeping its negative momentum.
In this way, we generalize the SSNM algorithm by changing the gradient estimator and adjusting the negative momentum correspondingly.

\begin{algorithm}[h]
\caption{Generalized SSNM }
\label{Algorithm1}
\begin{algorithmic}
\STATE{$\phi_1=\phi_2=\cdots=\phi_m=x^1$}
\STATE{$\pi_i=\frac{\sqrt{L_i}}{2\sum_{j=1}^m\sqrt{L_j}}+\frac{1}{2m}$}
\STATE{$\tau_i=\lambda/\pi_i$}
\FOR{$k=1,2,..., K$}
\STATE{Sample $i_k$ from $\{1,2,\cdots,m\}$ with probability $\{\pi_i\}_{i=1}^{m}$}
\STATE{$y_{i_{k}}^{k}=\tau_{i_{k}} x^{k}+(1-\tau_{i_{k}}) \phi_{i_{k}}^{k}$}
\STATE{$\widetilde{\nabla}^{k}=\frac{1}{\pi_{i_{k}}}\left(\nabla \hat{g}_{i_k}(y_{i_k}^k)-\nabla \hat{g}_{i_k}(\phi_{i_k}^k)\right)+  \sum_{i=1}^m\nabla \hat{g}_{i}(\phi_{i}^k)$}
\STATE{Perform a proximal/projection step\\ $x^{k+1}=\arg \min _{x\in \Gamma}\left\{h(x)+\left\langle\widetilde{\nabla}^{k}, x\right\rangle+\frac{1}{2 \eta}\left\|x^{k}-x\right\|^{2}\right\}$}
\STATE{Sample $j_k$ from $\{1,2,\cdots,m\}$ with probability $\{\pi_i\}_{i=1}^{m}$}
\STATE{$\phi_{j_{k}}^{k+1}=\tau_{j_{k}} x^{k+1}+(1-\tau_{j_{k}}) \phi_{j_{k}}^{k}$}
\ENDFOR
\STATE \textbf{return} $x^{k+1}$ 
\end{algorithmic}
\end{algorithm}


\begin{theorem} \label{thmssnm}
Generalized SSNM (Algorithm~\ref{Algorithm1}) requires no more than
$$O\left(\left(m+\frac{\sum_{j=1}^m\sqrt{L_j}}{\sqrt{\sum_{i=1}^m \mu_i}}\right)\log(1/\epsilon)\right)$$ gradient computations to reach an $\epsilon$-solution for Model~\eqref{pro} in expectation.
\end{theorem}


\begin{remark}
{\rm Comparing the upper bound in Theorem~\ref{thmssnm} with the lower bound in Theorem~\ref{lower-iter-bound}, we observe the upper and lower bounds agree, up to a logarithmic factor,
showing that Algorithm~\ref{Algorithm1} is essentially optimal for the finite-sum optimization model~\eqref{pro}.}
\end{remark}

\begin{remark}
{\rm The function $h(x)$ can be generalized to be a lower semi-continuous (possibly non-differentiable) function, as long as the proximal operator is available.}
\end{remark}


In the remainder of this section, we shall prove Theorem~\ref{thmssnm}, following similar steps as in SSNM~\cite{zhou2019direct} and Katyusha~\cite{allen2017katyusha} with necessary adaptations. To start, we note two lemmas. The first one is generalized from Lemma~2.4 in \cite{allen2017katyusha}.
\begin{lemma}{(variance upper bound)}\label{lemma1}
\[
\mathbb{E}_{i_{k}}\left[\left\|\widetilde{\nabla}^{k}- \sum_{i=1}^{m} \nabla \hat{g}_{i}\left(y_{i}^{k}\right)\right\|^{2}\right]
\leq \sum_{i=1}^{m} \frac{2L_{i}}{\pi_{i}} \left(\hat{g}_{i}\left(\phi_{i}^{k}\right)-\hat{g}_{i}\left(y_{i}^{k}\right) -\left\langle\nabla \hat{g}_{i}\left(y_{i}^{k}\right),\phi_{i}^{k}-y_{i}^{k}\right\rangle  \right).
\]
\end{lemma}

\begin{proof}
Since each $\hat{g}_i(x)$ is convex and $L_i$ smooth, Theorem 2.1.5 in \cite{nesterov2018lectures} stipulates
\begin{equation}
\left\|\nabla \hat{g}_{i}\left(y_{i}^{k}\right)-\nabla \hat{g}_{i}\left(\phi_{i}^{k}\right)\right\|^{2}\leq 2L_{i}\left(\hat{g}_{i}\left(\phi_{i}^{k}\right)-\hat{g}_{i}\left(y_{i}^{k}\right) -\left\langle\nabla \hat{g}_{i}\left(y_{i}^{k}\right),\phi_{i}^{k}-y_{i}^{k}\right\rangle  \right). \label{squineq}
\end{equation}
Observing
\begin{equation} \label{ex-ineq}
\mathbb{E}\|\zeta-\mathbb{E} \zeta\|^{2}=\mathbb{E}\|\zeta\|^{2}-\|\mathbb{E} \zeta\|^{2},
\end{equation}
we have
\begin{eqnarray*}
    & & \mathbb{E}_{i_{k}}\left[\left\|\widetilde{\nabla}^{k}- \sum_{i=1}^{m} \nabla \hat{g}_{i}\left(y_{i}^{k}\right)\right\|^{2}\right] \\
    &=& \mathbb{E}_{i_{k}}\left[\left\|\frac{1}{\pi_{i_{k}}}\left(\nabla \hat{g}_{i_{k}}\left(y_{i_{k}}^{k}\right)-\nabla \hat{g}_{i_{k}}\left(\phi_{i_{k}}^{k}\right)\right)- \sum_{i=1}^{m}\left(\nabla \hat{g}_{i}\left(y_{i}^{k}\right)-\nabla \hat{g}_{i}\left(\phi_{i}^{k}\right)\right) \right\|^{2}\right] \\
    & \overset{\eqref{ex-ineq}}{\leq} & \mathbb{E}_{i_{k}}\left[\left\|\frac{1}{\pi_{i_{k}}}\left(\nabla \hat{g}_{i_{k}}\left(y_{i_{k}}^{k}\right)-\nabla \hat{g}_{i_{k}}\left(\phi_{i_{k}}^{k}\right)\right) \right\|^{2}\right] \\ 
    & \overset{\eqref{squineq}}{\leq} & \mathbb{E}_{i_{k}}\left[\frac{2L_{i_{k}}}{ \pi_{i_{k}}^2} \left(\hat{g}_{i_{k}}\left(\phi_{i_{k}}^{k}\right)-\hat{g}_{i_{k}}\left(y_{i_{k}}^{k}\right) -\left\langle\nabla \hat{g}_{i_{k}}\left(y_{i_{k}}^{k}\right),\phi_{i_{k}}^{k}-y_{i_{k}}^{k}\right\rangle  \right)\right] \\ 
    & = & \sum_{i=1}^{m} \frac{2L_{i}}{\pi_{i}} \left(\hat{g}_{i}\left(\phi_{i}^{k}\right)-\hat{g}_{i}\left(y_{i}^{k}\right) -\left\langle\nabla \hat{g}_{i}\left(y_{i}^{k}\right),\phi_{i}^{k}-y_{i}^{k}\right\rangle \right).
\end{eqnarray*}
\end{proof}

The following lemma is identical to Lemma 3.5 in \cite{allen2017katyusha}.
\begin{lemma} \label{lemma-3}
Suppose $h(x)$ is $\mu$-strongly convex, and $x^{k+1}$ satisfies
\[
x^{k+1} := \arg \min _{x\in \Gamma}\left\{h(x)+\left\langle\widetilde{\nabla}^{k}, x\right\rangle+\frac{1}{2 \eta}\left\|x^{k}-x\right\|^{2}\right\}.
\]
Then, for all $u\in \Gamma$, it holds that
\begin{eqnarray*}
& & \left\langle\widetilde{\nabla}^{k}, x^{k+1}-u\right\rangle \\
& \leq & -\frac{1}{2 \eta}\left\|x^{k+1}-x^{k}\right\|^{2}+\frac{1}{2 \eta}\left\|x^{k}-u\right\|^{2}-\frac{1+\eta \mu}{2 \eta}\left\|x^{k+1}-u\right\|^{2}+h(u)-h\left(x^{k+1}\right).
\end{eqnarray*}
\end{lemma}

\begin{proof}
By the optimality condition, there exists an $h'\in \partial h(x)|_{x=x^{k+1}}$, such that
\[
\left\langle\frac{1}{\eta}(x^{k+1}-x^{k})+\widetilde{\nabla}^{k}+h',u-x^{k+1} \right\rangle\geq 0,\,\,\, \forall u \in \Gamma.
\]
Observe the identity $\left\langle x^{k+1}-x^{k},x^{k+1}-u \right\rangle=\frac 1 2\|x^{k+1}-x^{k}\|^2+\frac 1 2\|x^{k+1}-u\|^2-\frac 1 2\|x^{k}-u\|^2$ and the inequality $\left\langle h',u-x^{k+1} \right\rangle\leq h(u)-h\left(x^{k+1}\right)-\frac{\mu}{2}\|x^{k+1}-u\|^2$, the lemma follows.
\end{proof}

Now, we shall proceed to proving the theorem.

\begin{proof}[Proof of Theorem \ref{thmssnm}]
First, we assume that we can choose $\eta>0$ and $\lambda>0$ such that $0\leq \tau_{i} \leq 1$, $\frac{1}{\eta}-\frac{\sum_{i=1}^{m}\tau_i L_i}{m} \geq \frac{L_i\tau_i}{m \pi_i(1-\tau_i)}$, $\forall i=1,2,...,m$ hold. Later in the proof, we will show how to choose $\eta>0$ and $\lambda>0$ to satisfy these inequalities.

By the convexity of $\hat{g}_{i_{k}}(\cdot)$,
\begin{eqnarray*}
 & &    \frac{1}{\pi_{i_{k}}}\left(\hat{g}_{i_{k}}\left(y_{i_{k}}^{k}\right)-\hat{g}_{i_{k}}\left(x^{\star}\right)\right)
    \leq \frac{1}{\pi_{i_{k}}}\left\langle\nabla \hat{g}_{i_{k}}\left(y_{i_{k}}^{k}\right), y_{i_{k}}^{k}-x^{\star}\right\rangle \\
    &=& \frac{1}{\pi_{i_{k}}}\frac{1-\tau_{i_{k}}}{\tau_{i_{k}}}\left\langle\nabla \hat{g}_{i_{k}}\left(y_{i_{k}}^{k}\right), \phi_{i_{k}}^{k}-y_{i_{k}}^{k}\right\rangle+\left\langle\frac{1}{ \pi_{i_{k}}}\nabla \hat{g}_{i_{k}}\left(y_{i_{k}}^{k}\right)-\widetilde{\nabla}^{k}, x^{k}-x^{\star}\right\rangle\\
    & &+\left\langle\widetilde{\nabla}^{k}, x^{k}-x^{k+1}\right\rangle+\left\langle\widetilde{\nabla}^{k}, x^{k+1}-x^{\star}\right\rangle
\end{eqnarray*}
where the last step uses the definition that $y_{i}=\tau_{i} x^{k}+(1-\tau_{i}) \phi_{i}^{k}$, for $i=1,2,...,m$.

Observe
\[
\mathbb{E}_{i_{k}}\left[\widetilde{\nabla}^{k}-\frac{1}{\pi_{i_{k}}}\nabla \hat{g}_{i_{k}}\left(y_{i_{k}}^{k}\right)\right]=\mathbb{E}_{i_{k}}\left[\widetilde{\nabla}^{k}\right]-\sum_{i=1}^{m} \nabla \hat{g}_{i}\left(y_{i}^{k}\right)=0.
\]
By taking expectation with respect to sample $i_k$, it follows
\begin{eqnarray} \label{opt-gap} \\
& & \sum_{i=1}^{m} \left( \hat{g}_{i}\left(y_{i}^{k}\right)-\hat{g}_i\left(x^{\star}\right) \right)  \nonumber  \\
&\leq& \sum_{i=1}^{m}\frac{1-\tau_{i}}{\tau_{i}}\left\langle\nabla \hat{g}_{i}\left(y_{i}^{k}\right), \phi_{i}^{k}-y_{i}^{k}\right\rangle+\mathbb{E}_{i_{k}}\left[\left\langle\widetilde{\nabla}^{k}, x^{k}-x^{k+1}\right\rangle\right]+\mathbb{E}_{i_{k}}\left[\left\langle\widetilde{\nabla}^{k}, x^{k+1}-x^{\star}\right\rangle\right] . \nonumber
\end{eqnarray}



Theorem 2.1.5 in \cite{nesterov2018lectures} tells us that
\[
\hat{g}_{j_{k}}\left(\phi_{j_{k}}^{k+1}\right)-\hat{g}_{j_{k}}\left(y_{j_{k}}^{k}\right) \leq\left\langle\nabla \hat{g}_{j_{k}}\left(y_{j_{k}}^{k}\right), \phi_{j_{k}}^{k+1}-y_{j_{k}}^{k}\right\rangle+\frac{L_{j_{k}}}{2}\left\|\phi_{j_{k}}^{k+1}-y_{j_{k}}^{k}\right\|^{2} .
\]

Dividing the above inequality by $\pi_{j_{k}}\tau_{j_{k}}$ and taking expectation with respect to sample $j_k$, we obtain
\begin{eqnarray*}
& & \mathbb{E}_{j_{k}}\left[\frac{1}{\pi_{j_{k}}\tau_{j_{k}}}\hat{g}_{j_{k}}\left(\phi_{j_{k}}^{k+1}\right)\right]-\sum_{i=1}^{m} \frac{1}{\tau_i} \hat{g}_i\left(y_i^{k}\right) \\
&\leq& \left\langle \sum_{i=1}^{m} \nabla \hat{g}_{i}\left(y_{i}^{k}\right), x^{k+1}-x^{k}\right\rangle+\frac{\sum_{i=1}^{m}\tau_i L_i}{2}\|x^{k+1}-x^{k}\|^2,
\end{eqnarray*}
which is equivalent to
\begin{eqnarray*}
     \left\langle\widetilde{\nabla}^{k}, x^{k}-x^{k+1}\right\rangle
     &\leq&
    \sum_{i=1}^{m} \frac{1}{\tau_i} \hat{g}_{i}\left(y_{i}^{k}\right)- \mathbb{E}_{j_{k}}\left[\frac{1}{\pi_{j_{k}}\tau_{j_{k}}}\hat{g}_{j_{k}}\left(\phi_{j_{k}}^{k+1}\right)\right] \\
    & & +\left\langle \sum_{i=1}^{m} \nabla \hat{g}_{i}\left(y_{i}^{k}\right)-\widetilde{\nabla}^{k},
    x^{k+1}-x^{k}\right\rangle+\frac{\sum_{i=1}^{m}\tau_i L_i}{2}\|x^{k+1}-x^{k}\|^2.
\end{eqnarray*}
Taking expectation with respect to sample $i_k$ we get
\begin{eqnarray*}
     \mathbb{E}_{i_{k}}\left[\left\langle\widetilde{\nabla}^{k}, x^{k}-x^{k+1}\right\rangle\right] &\leq&
    \sum_{i=1}^{m} \frac{1}{\tau_i} \hat{g}_{i}\left(y_{i}^{k}\right)
    - \mathbb{E}_{i_{k},j_{k}}\left[\frac{1}{\pi_{j_{k}}\tau_{j_{k}}}\hat{g}_{j_{k}}\left(\phi_{j_{k}}^{k+1}\right)\right]\\
& & +\mathbb{E}_{i_{k}}\left[\left\langle \sum_{i=1}^{m}
\nabla \hat{g}_{i}\left(y_{i}^{k}\right)-\widetilde{\nabla}^{k}, x^{k+1}-x^{k}\right\rangle\right]+\frac{\sum_{i=1}^{m}\tau_i L_i}{2}\mathbb{E}_{i_{k}}\left[\|x^{k+1}-x^{k}\|^2\right] .
\end{eqnarray*}

By Lemma~\ref{lemma-3}, taking expectation we have
\begin{eqnarray*}
& & \mathbb{E}_{i_{k}}\left[ \left\langle\widetilde{\nabla}^{k}, x^{k+1}-x^\star\right\rangle \right]  \\
&\leq& - \mathbb{E}_{i_{k}}\left[\frac{1}{2 \eta}\left\|x^{k+1}-x^{k}\right\|^{2}\right]+\frac{1}{2 \eta}\left\|x^{k}-x^\star\right\|^{2}-\frac{1+\eta \mu}{2 \eta} \mathbb{E}_{i_{k}}\left[ \left\|x^{k+1}-x^\star\right\|^{2}\right]+h(x^\star) - \mathbb{E}_{i_{k}}\left[ h\left(x^{k+1}\right)\right] .
\end{eqnarray*}

Applying the above inequalities on \eqref{opt-gap}, we have 
\begin{eqnarray*}
    & & \sum_{i=1}^{m} \left( \hat{g}_{i}\left(y_{i}^{k}\right)-\hat{g}_i \left(x^{\star}\right) \right) \\
    &\leq&
     \sum_{i=1}^{m}\frac{1-\tau_{i}}{\tau_{i}}\left\langle\nabla \hat{g}_{i}\left(y_{i}^{k}\right), \phi_{i}^{k}-y_{i}^{k}\right\rangle
    +\sum_{i=1}^{m} \frac{1}{\tau_i} \hat{g}_{i}\left(y_{i}^{k}\right)
    - \mathbb{E}_{i_{k},j_{k}}\left[\frac{1}{\pi_{j_{k}}\tau_{j_{k}}}\hat{g}_{j_{k}}\left(\phi_{j_{k}}^{k+1}\right) \right] \\
   & & +\mathbb{E}_{i_{k}}\left[\left\langle \sum_{i=1}^{m} \nabla \hat{g}_{i}\left(y_{i}^{k}\right)-\widetilde{\nabla}^{k}, x^{k+1}-x^{k}\right\rangle\right]+\frac{\sum_{i=1}^{m}\tau_i L_i}{2}\mathbb{E}_{i_{k}}\left[\|x^{k+1}-x^{k}\|^2\right]\\
   & & -\frac{1}{2 \eta} \mathbb{E}_{i_{k}}\left[\left\|x^{k+1}-x^{k}\right\|^{2}\right]+\frac{1}{2 \eta}\left\|x^{k}-x^{\star}\right\|^{2}-\frac{1+\eta \mu}{2 \eta} \mathbb{E}_{i_{k}}\left[\left\|x^{k+1}-x^{\star}\right\|^{2}\right]
    +h\left(x^{\star}\right)-\mathbb{E}_{i_{k}}\left[h\left(x^{k+1}\right)\right]\\
   & \leq&  \sum_{i=1}^{m}\frac{1-\tau_{i}}{\tau_{i}}\left\langle\nabla \hat{g}_{i}\left(y_{i}^{k}\right), \phi_{i}^{k}-y_{i}^{k}\right\rangle
    +\sum_{i=1}^{m} \frac{1}{\tau_i} \hat{g}_{i}\left(y_{i}^{k}\right)
    - \mathbb{E}_{i_{k},j_{k}}\left[\frac{1}{\pi_{j_{k}}\tau_{j_{k}}}\hat{g}_{j_{k}}\left(\phi_{j_{k}}^{k+1}\right)\right]\\
   & & + \frac{1}{2\left(\frac{1}{\eta}-\sum_{i=1}^{m}\tau_i L_i \right)}\mathbb{E}_{i_{k}}\left[\left\| \sum_{i=1}^{m} \nabla \hat{g}_{i}\left(y_{i}^{k}\right)-\widetilde{\nabla}^{k}\right\|^2\right]+\frac{1}{2 \eta}\left\|x^{k}-x^{\star}\right\|^{2}\\
   & & -\frac{1+\eta \mu}{2 \eta} \mathbb{E}_{i_{k}}\left[\left\|x^{k+1}-x^{\star}\right\|^{2}\right]+h\left(x^{\star}\right)-\mathbb{E}_{i_{k}}\left[h\left(x^{k+1}\right)\right]\\
   &\leq&  \sum_{i=1}^{m}\frac{1-\tau_{i}}{\tau_{i}}\left\langle\nabla \hat{g}_{i}\left(y_{i}^{k}\right), \phi_{i}^{k}-y_{i}^{k}\right\rangle
    +\sum_{i=1}^{m} \frac{1}{\tau_i} \hat{g}_{i}\left(y_{i}^{k}\right)
    - \mathbb{E}_{i_{k},j_{k}}\left[\frac{1}{\pi_{j_{k}}\tau_{j_{k}}}\hat{g}_{j_{k}}\left(\phi_{j_{k}}^{k+1}\right)\right]\\
   & &+ \frac{1}{2\left(\frac{1}{\eta}-\sum_{i=1}^{m}\tau_i L_i \right)}\sum_{i=1}^{m} \frac{2L_{i}}{\pi_{i}} \left(\hat{g}_{i}\left(\phi_{i}^{k}\right)-\hat{g}_{i}\left(y_{i}^{k}\right) -\left\langle\nabla \hat{g}_{i}\left(y_{i}^{k}\right),\phi_{i}^{k}-y_{i}^{k}\right\rangle  \right)\\
   & &+\frac{1}{2 \eta}\left\|x^{k}-x^{\star}\right\|^{2}-\frac{1+\eta \mu}{2 \eta} \mathbb{E}_{i_{k}}\left[\left\|x^{k+1}-x^{\star}\right\|^{2}\right]+h\left(x^{\star}\right)-\mathbb{E}_{i_{k}}\left[h\left(x^{k+1}\right)\right]\\
   &\leq&  \sum_{i=1}^{m}\frac{1-\tau_{i}}{\tau_{i}}\left\langle\nabla \hat{g}_{i}\left(y_{i}^{k}\right), \phi_{i}^{k}-y_{i}^{k}\right\rangle
    +\sum_{i=1}^{m} \frac{1}{\tau_i} \hat{g}_{i}\left(y_{i}^{k}\right)
    - \mathbb{E}_{i_{k},j_{k}}\left[\frac{1}{\pi_{j_{k}}\tau_{j_{k}}}\hat{g}_{j_{k}}\left(\phi_{j_{k}}^{k+1}\right)\right]\\
  &  &+ \sum_{i=1}^{m} \frac{1-\tau_{i}}{\tau_{i}} \left(\hat{g}_{i}\left(\phi_{i}^{k}\right)-\hat{g}_{i}\left(y_{i}^{k}\right) -\left\langle\nabla \hat{g}_{i}\left(y_{i}^{k}\right),\phi_{i}^{k}-y_{i}^{k}\right\rangle\right)\\
  &  &+\frac{1}{2 \eta}\left\|x^{k}-x^{\star}\right\|^{2}-\frac{1+\eta \mu}{2 \eta} \mathbb{E}_{i_{k}}\left[\left\|x^{k+1}-x^{\star}\right\|^{2}\right]+h\left(x^{\star}\right)-\mathbb{E}_{i_{k}}\left[h\left(x^{k+1}\right)\right]
\end{eqnarray*}
where the second inequality follows from the inequality $\langle a, b\rangle \leq \frac{1}{2 \beta}\|a\|^{2}+\frac{\beta}{2}\|b\|^{2}$ and $\frac{1}{\eta}-\sum_{i=1}^{m}\tau_i L_i >0$, and the third inequality is due to Lemma \ref{lemma1}, and the last inequality holds by the assumption that $\frac{1}{\eta}-\sum_{i=1}^{m}\tau_i L_i \geq \frac{L_i\tau_i}{\pi_i(1-\tau_i)}$.

Canceling out and rearranging the terms, we have
\begin{eqnarray*}
    &     &  \mathbb{E}_{i_{k},j_{k}}\left[\frac{1}{\pi_{j_{k}}\tau_{j_{k}}} \hat{g}_{j_{k}}\left(\phi_{j_{k}}^{k+1}\right)\right]-F(x^{\star})\\
    &\leq & \sum_{i=1}^{m} \frac{1-\tau_{i}}{\tau_{i}} \hat{g}_{i}\left(\phi_{i}^{k}\right)+\frac{1}{2 \eta}\left\|x^{k}-x^{\star}\right\|^{2}-\frac{1+\eta \mu}{2 \eta} \mathbb{E}_{i_{k}}\left[\left\|x^{k+1}-x^{\star}\right\|^{2}\right]-\mathbb{E}_{i_{k}}\left[h\left(x^{k+1}\right)\right].
\end{eqnarray*}

By the convexity of $h(\cdot)$ and $\phi_{j_{k}}^{k+1}=\tau_{j_{k}} x^{k+1}+(1-\tau_{j_{k}}) \phi_{j_{k}}^{k}$, as long as $0\leq \tau_{j_{k}} \leq 1$, we have
\[
h(\phi_{j_{k}}^{k+1})\leq \tau_{j_{k}} h(x^{k+1})+(1-\tau_{j_{k}}) h(\phi_{j_{k}}^{k}) .
\]
Dividing the above inequality by $m \pi_{j_{k}} \tau_{j_{k}}$ and taking expectation with respect to sample $j_k$ and sample $i_k$, we obtain
\[
-\mathbb{E}_{i_{k}}\left[h\left(x^{k+1}\right)\right] \leq  \sum_{i=1}^{m} \frac{1-\tau_{j_{k}}}{m \tau_{j_{k}}} h\left(\phi_{i}^{k}\right)- \mathbb{E}_{i_{k}, j_{k}}\left[\frac{1}{m \pi_{j_{k}} \tau_{j_{k}}}     h\left(\phi_{j_{k}}^{k+1}\right)\right] .
\]
Applying the above inequality and using $F_{i}(\cdot):=\hat{g}_{i}(\cdot)+\frac 1 m h(\cdot)$, we further derive
\begin{eqnarray*}
    &      &  \mathbb{E}_{i_{k},j_{k}}\left[\frac{1}{\pi_{j_{k}}\tau_{j_{k}}}F_{j_{k}}\left(\phi_{j_{k}}^{k+1}\right)\right]-F(x^{\star})\\
    &\leq  & \sum_{i=1}^{m} \frac{1-\tau_{i}}{\tau_{i}} F_{i}\left(\phi_{i}^{k}\right)+\frac{1}{2 \eta}\left\|x^{k}-x^{\star}\right\|^{2}-\frac{1+\eta \mu}{2 \eta} \mathbb{E}_{i_{k}}\left[\left\|x^{k+1}-x^{\star}\right\|^{2}\right] .
\end{eqnarray*}

Adding $\mathbb{E}_{i_{k},j_{k}}\left[\sum_{i\neq j_k}\frac{1}{\pi_i\tau_i}F_i\left(\phi_i^{k}\right)\right]$ to both sides, we have
\begin{eqnarray*}
    &     & \mathbb{E}_{i_{k},j_{k}}\left[\sum_{i=1}^m\frac{1}{\pi_i\tau_i}\left(F_i\left(\phi_i^{k+1}\right)-F_i(x^{\star})\right)\right]\\
    &\leq & \sum_{i=1}^{m} \frac{1-\pi_i\tau_{i}}{\pi_i\tau_{i}} \left(F_{i}\left(\phi_{i}^{k}\right)-F_i(x^{\star})\right)+\frac{1}{2 \eta}\left\|x^{k}-x^{\star}\right\|^{2}-\frac{1+\eta \mu}{2 \eta} \mathbb{E}_{i_{k}} \left[\left\|x^{k+1}-x^{\star}\right\|^{2} \right] .
\end{eqnarray*}
Noticing $\sum_{i=1}^{m} \left(F_{i}\left(\phi_{i}^{k}\right)-F_i(x^{\star})\right)$ may not be positive, we need to add the following term in our Lyapunov function
\begin{eqnarray*}
  & &    -\sum_{i=1}^{m}\frac{1}{\pi_i\tau_i}\left\langle\nabla F_{i}\left(x^{\star}\right), \phi_{i}^{k+1}-x^{\star}\right\rangle \\
  &=&  -\frac{1}{\pi_{j_{k}} \tau_{j_{k}}}\left\langle\nabla F_{j_{k}}\left(x^{\star}\right), \phi_{j_{k}}^{k+1}-x^{\star}\right\rangle-
  \sum_{i \neq j_k } 
  \frac{1}{\pi_i\tau_i}\left\langle\nabla F_{i}\left(x^{\star}\right), \phi_{i}^{k}-x^{\star}\right\rangle\\
  &=& -\frac{1}{\pi_{j_{k}}}\left\langle\nabla F_{j_{k}}\left(x^{\star}\right), x^{k+1}-x^{\star}\right\rangle
    +\frac{1}{\pi_{j_{k}}}\left\langle\nabla F_{j_{k}}\left(x^{\star}\right), \phi_{j_{k}}^{k}-x^{\star}\right\rangle\\
  &  & - \sum_{i =1}^{m}\frac{1}{\pi_i\tau_i}\left\langle\nabla F_{i}\left(x^{\star}\right), \phi_{i}^{k}-x^{\star}\right\rangle .
\end{eqnarray*}

Taking expectation with respect to samples $i_k$ and $j_k$ and using $\pi_i\tau_i=\lambda$, we get
\[
\mathbb{E}_{i_{k},j_{k}}\left[-\frac{1}{\lambda} \sum_{i=1}^{m}\left\langle\nabla F_{i}\left(x^{\star}\right), \phi_{i}^{k+1}-x^{\star}\right\rangle\right]=
-\frac{1-\lambda}{\lambda} \sum_{i =1}^{m}\left\langle\nabla F_{i}\left(x^{\star}\right), \phi_{i}^{k}-x^{\star}\right\rangle .
\]

Denoting $D^k:=\sum_{i =1}^{m}F_{i}\left(\phi_{i}^{k}\right)-F\left(x^{\star}\right)-\sum_{i =1}^{m}\left\langle\nabla F_{i}\left(x^{\star}\right), \phi_{i}^{k}-x^{\star}\right\rangle$ and $P^{k} :=\left\|x^{k}-x^{\star}\right\|^{2}$, we have
\[
\frac{1}{\lambda}\mathbb{E}_{i_{k},j_{k}}\left[D^{k+1}\right]+\frac{1+\eta \mu}{2 \eta}\mathbb{E}_{i_{k}}\left[P^{k+1}\right]\leq \frac{1-\lambda}{\lambda} D^k+\frac{1}{2 \eta} P^k.
\]

Finally, it remains to choose $\eta$ and $\lambda$ so as to satisfy
 \[
 0\leq \tau_{i} \leq 1, \mbox{ and }  \frac{1}{\eta}-\sum_{i=1}^{m}\tau_i L_i \geq \frac{L_i\tau_i}{\pi_i(1-\tau_i)},\,\,\, i=1,2,...,m.
 \]
We consider two cases separately.

\textbf{Case I.} If $\sqrt{\mu}\leq \frac{\sum_{j=1}^m\sqrt{L_j}}{m}$, then we choose
\[
\lambda := \frac{\sqrt{\mu}}{4\sum_{j=1}^m\sqrt{L_j}}, \quad \eta := \frac{1}{4\sqrt{\mu}\sum_{j=1}^{m}\sqrt{L_j} }.
\]

Recall
$\pi_i=\frac{\sqrt{L_i}}{2\sum_{j=1}^m\sqrt{L_j}}+\frac{1}{2m}$ for all $1\le i\le m$. Therefore,
\[
\tau_i=\frac{\lambda}{\pi_i}=\frac{\sqrt{\mu}}{2\sqrt{L_i}+\frac 2 m \sum_{j=1}^m\sqrt{L_j}}\leq \frac{\sqrt{\mu}}{\frac 2 m \sum_{j=1}^m\sqrt{L_j}}\leq \frac 1 2
\]
and
\begin{eqnarray*}
     \sum_{j=1}^{m}\tau_j L_j +\frac{L_i\tau_i}{\pi_i(1-\tau_i)} &\leq& \sum_{j=1}^{m}L_j\frac{\sqrt{\mu}}{2\sqrt{L_j}}+\frac{2\lambda L_i}{ \pi_i^2}\\
    &\leq& \frac{\sqrt{\mu}\sum_{j=1}^{m}\sqrt{L_j} }{2}+\frac{8(\sum_{j=1}^{m}\sqrt{L_j})^2\lambda L_i}{L_i}\\
    &\leq& 3\sqrt{\mu}\sum_{j=1}^{m}\sqrt{L_j} .
\end{eqnarray*}

Since $\eta=\frac{1}{4\sqrt{\mu}\sum_{j=1}^{m}\sqrt{L_j} }$, we know $\frac{1}{\eta}-\sum_{j=1}^{m}\tau_j L_j \geq \frac{L_i\tau_i}{ \pi_i(1-\tau_i)}$ and
\[
\frac{1}{\lambda}\geq (1+\eta \mu)\frac{1-\lambda}{\lambda} .
\]
Thus, in this case we obtain
\[
\frac{1}{\lambda}\mathbb{E}_{i_{k},j_{k}}\left[D^{k+1}\right]+\frac{1+\eta \mu}{2 \eta}\mathbb{E}_{i_{k}}\left[P^{k+1}\right]\leq (1+\eta \mu)^{-1}\left( \frac{1}{\lambda} D^k+\frac{1+\eta \mu}{2 \eta} P^k\right).
\]

Telescoping the above inequalities, we have
\begin{eqnarray*}
& & \mathbb{E}\left[\left\|x^{K+1}-x^{\star}\right\|^{2}\right] \\
&\leq& \left(1+\frac{\sqrt{\mu}}{4\sum_{j=1}^m\sqrt{L_j}}\right)^{-K} \cdot\left(\frac{2\eta}{\lambda(1+\eta \mu)}\left(F\left(x^{1}\right)-F\left(x^{\star}\right)\right)+\left\|x^{1}-x^{\star}\right\|^{2}\right) .
\end{eqnarray*}

\textbf{Case II.} If $\sqrt{\mu}> \frac{\sum_{j=1}^m\sqrt{L_j}}{m}$, then we choose
\[
\lambda := \frac{1}{4m}, \quad \eta := \frac{1}{4\mu m}.
\]

In this case, we have
\[
\tau_i=\frac{\lambda}{\pi_i}=\frac{\frac{1}{2m}\sum_{j=1}^m\sqrt{L_j}}{\sqrt{L_i}+\frac 1 m \sum_{j=1}^m\sqrt{L_j}}\leq \frac 1 2
\]
and
\begin{eqnarray*}
    \sum_{j=1}^{m}\tau_j L_j +\frac{L_i\tau_i}{\pi_i(1-\tau_i)}&\leq& \frac{1}{2m}\left(\sum_{j=1}^m\sqrt{L_j}\right)^2+\frac{2\lambda L_i}{ \pi_i^2}\\
    &\leq& \frac{\left(\sum_{j=1}^m\sqrt{L_j}\right)^2}{2m}+\frac{8\left(\sum_{j=1}^{m}\sqrt{L_j}\right)^2\lambda L_i}{L_i}\\
    &\leq& \frac{3\left(\sum_{j=1}^{m}\sqrt{L_j}\right)^2 }{m}\\
    &\leq& 3 \mu m.
\end{eqnarray*}

Since $\eta=\frac{1}{4\mu m}$, we have $\frac{1}{\eta}-\sum_{j=1}^{m}\tau_j L_j \geq \frac{L_i\tau_i}{\pi_i(1-\tau_i)}$, and
\[
\frac{1}{\lambda}\geq (1+\eta \mu)\frac{1-\lambda}{\lambda} .
\]

As a consequence,
\[
\frac{1}{\lambda}\mathbb{E}_{i_{k},j_{k}}\left[D^{k+1}\right]+\frac{1+\eta \mu}{2 \eta}\mathbb{E}_{i_{k}}\left[P^{k+1}\right]\leq (1+\eta \mu)^{-1}\left( \frac{1}{\lambda} D^k+\frac{1+\eta \mu}{2 \eta} P^k\right).
\]
Taking expectations and telescoping the above inequalities, we have
\begin{equation}
\mathbb{E}\left[\left\|x^{K+1}-x^{\star}\right\|^{2}\right] \leq\left(1+\frac{1}{4m}\right)^{-K} \cdot\left(\frac{2\eta}{\lambda(1+\eta \mu)}\left(F\left(x^{1}\right)-F\left(x^{\star}\right)\right)+\left\|x^{1}-x^{\star}\right\|^{2}\right) .
\end{equation}
Summarizing Case I and Case II, the theorem is proven.
\end{proof}


\section{The Dual of Finite Sum Model and Its Complexity Status} \label{sec4}

The conjugate of a convex function $g$ defined over $\mathbb{R}^{n}$ 
is defined as
\[
g^*(y):=\max_{x\in \mathbb{R}^{n}}\left[y^{\top}x-g(x)\right] .
\]




It is well known (cf.~e.g.~Theorem 16.4 \cite{rockafellar2015convex}) that the conjugate of a finite sum of convex functions is given as:
\[
{\left(g_{1}+ \cdots + g_{m}\right)^{*}(p)=} \\
\min _{p_{1}+\cdots+p_{m}=p}
    \sum_{i=1}^{m} g_{i}^{*}\left(p_{i}\right)
\]
where $g_1, g_2, ... , g_m$ are proper convex functions which share at least a interior point in their respective domains.

Therefore, the dual of Model~\eqref{pro} is
\begin{equation} \label{FD}
\begin{array}{ll}
\min & g_1^* (p_1) + g_2^* (p_2) + \cdots + g_m^* (p_m) \\
\mbox{s.t.} & p_1 + p_2 + \cdots + p_m = 0 .
\end{array}
\end{equation}

To connect the dual model~\eqref{FD} with the original primal model~\eqref{pro}, let us note the following relationship.

\begin{lemma}{(Theorem 1 \cite{zhou2018fenchel})}\label{dual}
If $g$ is closed and strong convex with parameter $\mu$, then $g^*$ has a Lipschitz continuous gradient with
parameter $\frac{1}{\mu}$; if $g$ is closed and has a Lipschitz continuous gradient with parameter $L$, then $g^*$ is strong convex with parameter $\frac{1}{L}$.
\end{lemma}

Next, we will show the relationship between the primal and dual models. Consider the KKT conditions for the dual problem,
\[
\begin{cases}\nabla g_{i}^{*}\left(p_i^{\star}\right)-x^{\star}=0, & i=1, \ldots, m \\ p_1^{\star}+p_2^{\star}+\cdots+ p_m^{\star}=0.
\end{cases}
\]

It is well know that $\nabla g_{i}^{*}\left(p_i^{\star}\right)=x^{\star}$ is equivalent to $p_i^{\star}=\nabla g_{i}\left(x^{\star}\right)$. Therefore, $x^{\star}=\nabla g_{i}^{*}\left(p_i^{\star}\right)$ is the optimal solution for the primal problem. So after getting an approximation of dual solution $\{p_i^K\}_{i=1}^{m}$, we can define $x^K=\nabla g_{1}^{*}\left(p_1^{K}\right)$ to recover a primal solution.

From Lemma \ref{dual}, we know that $g_{1}^{*}$ is $\frac{1}{L_1}$ strongly convex and $\frac{1}{\mu_1}$-smooth. Therefore, we have the following lemma.
\begin{lemma}
Let $x^K=\nabla g_{1}^{*}\left(p_1^{K}\right)$, which recovers a primal solution satisfying
\[
\frac{1}{L_1}\|p_1^K-p_1^\star\|\leq \|x^K-x^\star\|\leq \frac{1}{\mu_1}\|p_1^K-p_1^\star\|.
\]
\end{lemma}

Therefore, we can recover a primal solution by using an approximative dual solution. If we can easily evaluate the gradient of $g_i^*(x), i=1,2,..., m$, then it is reasonable to solve the dual problem by using some first order methods.
To be able to compare the dual model with the primal in a directly comparable manner, let us simply consider
\begin{equation} \label{FD2}
\begin{array}{ll}
\min & G(p) := g_1 (p_1) + g_2 (p_2) + \cdots + g_m (p_m) \\
\mbox{s.t.} & p_1 + p_2 + \cdots + p_m = 0
\end{array}
\end{equation}
where $g_i$ is $\mu_i$-strongly convex and $L_i$-smooth.


\subsection{A Lower Bound on the Gradient Complexity for Model~\eqref{FD2}}

Without loss of generality, suppose $\frac{L_1}{\mu_1}\geq \frac{L_2}{\mu_2}\geq \cdots\geq \frac{L_n}{\mu_n}>0$. Assume we can only evaluate one $\nabla g_i$ at each iteration. To construct an adversarial example, we shall use Nesterov's construction and we divide $p_1,p_2,\cdots, p_m$ into $\lfloor\frac{m}{2}\rfloor$ blocks respectively, with all of them being in $ \ell_{2}$; that is, $p_i=(p_{i,1},p_{i,2},\cdots,p_{i,\lfloor \frac{m}{2}\rfloor})$ and $p_{i,j}=\{p_{i,j,l}\}_{l=1}^{\infty}\in \ell_{2}$.
Our idea is to construct an adversarial example such that the optimal solutions of each function $g_i(p_i)$ satisfy the constraint. So we no longer need to be concerned with the constraint when estimating the lower bound.

Let us consider the pair of functions $g_{2i-1}(p_{2i-1})$ and $g_{2i}(p_{2i})$ together.  
For $i=1,2,...,\lfloor \frac{m}{2}\rfloor$, construct
\begin{eqnarray*}
g_{2i-1}(p_{2i-1}) &=& \frac{(L_{2i}-\mu_{2i})\mu_{2i-1}}{8\mu_{2i}}\left(p_{2i-1,i}^{\top}A p_{2i-1,i}-2\gamma_i p_{2i-1,i,1}\right)+\frac{\mu_{2i-1}}{2}\|p_{2i-1}\|^2 \\
g_{2i}(p_{2i})     &=& \frac{L_{2i}-\mu_{2i}}{8}\left(p_{2i,i}^{\top}A p_{2i,i}+2\gamma_i p_{2i,i,1}\right)+\frac{\mu_{2i}}{2}\|p_{2i}\|^2
\end{eqnarray*}
where $A$ is defined as in \eqref{matrix-A}, and $\gamma_i$ is a constant and will be determined later. If $m$ is odd then we also construct
\[
g_{m}(p_{m})=\frac{\mu_{m}}{2}\|p_{m}\|^2.
\]

Denote $q_i := \frac{\sqrt{\kappa_{2i}}-1}{\sqrt{\kappa_{2i}}+1}$, where $\kappa_{2i}:=\frac{L_{2i}}{\mu_{2i}}$, $i=1,2,...,\lfloor \frac{m}{2}\rfloor$. Lemma \ref{l1} gives us
\begin{eqnarray*}
p^\star_{2i-1,i} &:=& \gamma_i(q_i,q_i^2,\cdots)^{\top}, \\
p^\star_{2i,i}   &:=& -\gamma_i(q_i,q_i^2,\cdots)^{\top}, \\
p^\star_{2i-1,j} &=& p^\star_{2i,j} :=\textbf{0}, \forall j\neq i, \\
p^\star_m &:=& \textbf{0}, \text{ if $m$ is odd}
\end{eqnarray*}
which minimize $g_i$, and the solutions $p^\star_i$'s happen to satisfy the constraint as well. Therefore, it is the optimal solution for Model~\eqref{FD2}.

Suppose the initial point is the origin. Let $K_i$ be the number of queries for $\nabla g_{2i-1}$ and $\nabla g_{2i}$. Then the total number of queries for gradients will be $K\geq K_1+K_2+\cdots+K_{\lfloor m/2\rfloor}$. Note that in each iteration, when computing $\nabla g_{2i-1}$ or $\nabla g_{2i}$, only the $i$th block extends one coordinate that is possibly nonzero. Therefore, only the first $K_{i}$ coordinates of the $i$th block of each variable can be nonzero. By defining $\gamma_i=\frac{\sqrt{1-q_i^2}}{q_i}\sqrt{\frac{2}{\mu_{2i}}}$,
we have $\|p_{2i,i}^0-p_{2i,i}^{\star}\|^2=2/\mu_{2i}$, $i=1,2, ... , \lfloor\frac{m}{2}\rfloor$. Using Lemma~\ref{l2}, we have
\begin{eqnarray*}
    G(p^K) - G(p^{\star}) &\geq & \sum_{i=1}^{\lfloor m/2\rfloor } \left( g_{2i}(p_{2i}^{K})-g_{2i}(p_{2i}^{\star}) \right) \\
    &\geq & \sum_{i=1}^{\lfloor m/2\rfloor }\frac{\mu_{2i}}{2} \left(\frac{\sqrt{\kappa_{2i}}-1}{\sqrt{\kappa_{2i}}+1}\right)^{2K_i}\|p_{2i,i}^0-p_{2i,i}^{\star}\|^2\\
    &\geq & \frac{\mu_{2i}}{2} \left(\frac{\sqrt{\kappa_{2i}}-1}{\sqrt{\kappa_{2i}}+1}\right)^{2K_i}\|p_{2i,i}^0-p_{2i,i}^{\star}\|^2\\
    & = & \left(\frac{\sqrt{\kappa_{2i}}-1}{\sqrt{\kappa_{2i}}+1}\right)^{2K_i}
\end{eqnarray*}
where the last step is because $\|p_{2i,i}^0-p_{2i,i}^{\star}\|^2=2/\mu_{2i}$.

To ensure $G(p^K) - G(p^{\star})<\epsilon$ we need
\[
K_j\geq \frac 1 2 \frac{1}{\log\left(\frac{\sqrt{\kappa_j}+1}{\sqrt{\kappa_j}-1}\right)}\log\left(\frac{1}{\epsilon}\right)
\ge \frac{\sqrt{\kappa_i}-1}{4} \log\left(\frac{1}{\epsilon}\right).
\]


We call a solution to Model~\eqref{FD2} to be an $\epsilon$ optimal solution if the constraint violation of the solution is no more than $\epsilon$, and its objective is no more than $\epsilon$ away from the true optimal value. Thus, we have the following theorem:
\begin{theorem} \label{LB-2}
For any first order method aiming at solving Model~\eqref{FD2}, assuming $\kappa_i$'s are uniformly bounded below from 1, then one needs at least
\begin{eqnarray*}
& & \Omega\left(\left((\sqrt{\kappa_2}-1)+(\sqrt{\kappa_3}-1)+\cdots+(\sqrt{\kappa_m}-1)\right)\log(1/\epsilon)\right) \\
&=&
\Omega\left(\left(\sqrt{\frac{L_2}{\mu_2}}+\sqrt{\frac{L_3}{\mu_3}}\cdots+\sqrt{\frac{L_m}{\mu_m}}-m+1\right)\log(1/\epsilon)\right)
\end{eqnarray*}
gradient computations to find an $\epsilon$-optimal solution.
\end{theorem}

\subsection{An Upper Bound on the Gradient Complexity for Model~\eqref{FD2}}

To eliminate the constraints, let us rewrite Model~\eqref{FD2} into the following form:
\[
\min _{p_{1}+\cdots+p_{m}=0} \sum_{i=1}^{m} g_{i}\left(p_{i}\right) = \min_{p_2,...,p_m} g_{1}\left(-p_2-\cdots-p_m\right)+\sum_{i=2}^{m} g_{i}\left(p_{i}\right) ,
\]
for which the so-called multi-block coordinate descent approach is well suited. In particular, the so-called accelerated randomized coordinate descent method~\cite{lu2015complexity} is applicable. To apply the result to that setting, we note that the gradient of each function is block-wise Lipschitz continuous with constants $L_1+L_i$. The function has convexity parameter $\min_i\mu_i/(L_1+L_i)$ with respect to the norm
$\|\cdot\|_{(L_1+L_2,\cdots,L_1+L_m)}$.
By using accelerated randomized coordinate descent, the iteration complexity is
\[
O\left(m\max_{i\geq 2}\sqrt{\frac{L_1+L_i}{\mu_i}}\log(1/\epsilon)\right)
\]
which is an upper bound for the iteration complexity for Model~\eqref{FD2}. Since we may choose to eliminate any $p_j$, 
we have the following theorem:
\begin{theorem} \label{UB-2}
To get an $\epsilon$-optimal solution, we need at most
\[
O\left(m\min_j\max_{i\neq j}\sqrt{\frac{L_j+L_i}{\mu_i}}\log(1/\epsilon)\right)
\]
iterations, and at each iteration, one needs to evaluate twice the gradient value $\nabla g_i(\cdot)$.
\end{theorem}

We observe that the upper bound in Theorem~\ref{UB-2} and the lower bound in Theorem~\ref{LB-2} are not equal, although they are in the same order of magnitude. However, in the case $L_1=L_2=\cdots=L_m$ and $\mu_1=\mu_2=\cdots=\mu_m$, and $L_1/\mu_1-1=\Omega(1)$ then the bounds match and they are optimal.
Although Model~\eqref{pro} and Model~\eqref{FD2} are duality in form, their complexity statuses appear to be different. It is interesting to note though, that the upper bound in Theorem~\ref{thmssnm} for Model~\eqref{pro} is better than the upper bound in Theorem~\ref{UB-2} when $L_i \equiv L$ and $\mu_i \equiv \mu$ for all $i=1,2,...,m$. To be sure, the oracles in evaluating the gradients of a convex function and in evaluating its conjugate are quite different and {\it not}\/ comparable in general. However, it pays to investigate what happens if we attempt to solve the dual formulation by its dual, namely the primal formulation Model~\eqref{pro}. Next section presents such a study.



\section{Multi-block Convex Minimization} \label{sec5}

Let us consider the following multi-block convex minimization model:
\begin{equation} \label{BlockOpt}
\begin{array}{lll}
\phi(b) := & \min           & f_{1}\left(y_{1}\right)+f_{2}\left(y_{2}\right)+\cdots+f_{m}\left(y_{m}\right) \\
           & \text { s.t. } & A_{1} y_{1}+A_{2} y_{2}+\cdots+A_{m} y_{m}=b \\
\end{array}
\end{equation}
where $f_i$ is $\mu_i$-strongly convex and $L_i$-smooth, $i=1,2,...,m$. The above model has found a wide range of applications. A popular approach for solving Model~\eqref{BlockOpt} has been the so-called Alternating Direction Method of Multipliers (ADMM) and its variants. It is known that ADMM can be made to converge linearly under some strongly convexity conditions (see~\cite{lin2015global,hong2017linear} and the references therein). Even without any strong convexity assumptions, the ADMM type methods can be modified to converge with an iteration complexity $O(1/\epsilon)$ to find an $\epsilon$-optimal solution, where {\it $\epsilon$-optimal solution}\/ refers to the solution with both the norm of the constraint violation and the function value to the optimal value to be less than $\epsilon$; see e.g.~\cite{lin2016iteration}. We shall show below that through solving the dual of Model~\eqref{BlockOpt}, it is possible to find an $\epsilon$-optimal solution for \eqref{BlockOpt}
with an iteration complexity of no more than $O\left( \frac{\log (1/\epsilon)}{\sqrt{\epsilon}}\right)$, without any strong convexity assumptions.

%

First of all, let us compute the dual of Model~\eqref{BlockOpt}.
\begin{lemma}
The conjugate of $\phi(b)$ as defined in Model~\eqref{BlockOpt} is
\[
\phi^{*}(x)=\sum_{i=1}^{m} f_i^*(A_i^{\top}x).
\]
\end{lemma}

\begin{proof}
By definition,
\begin{eqnarray*}
\phi^{*}(x) &=& \max_b
\left\{
\begin{array}{lll}
 \left\langle x,b\right\rangle - & \min & f_{1}\left(y_{1}\right)+f_{2}\left(y_{2}\right)+\cdots+f_{m}\left(y_{m}\right) \\
 & \text { s.t. } & A_{1} y_{1}+A_{2} y_{2}+\cdots+A_{m} y_{m}=b
 \end{array}
 \right\}
 \\
&=&\max_b\max_{A_{1} y_{1}+A_{2} y_{2}+\cdots+A_{m} y_{m}=b}\left(\left\langle x,b\right\rangle-
f_{1}\left(y_{1}\right)-f_{2}\left(y_{2}\right)-\cdots-f_{m}\left(y_{m}\right)\right)\\
&=&\max_{y_1,y_2,\cdots,y_m}\sum_{i=1}^{m}\left\langle A_i y_i,x\right\rangle- \sum_{i=1}^{m}f_i(y_i)\\
&=&\sum_{i=1}^{m} f_i^*(A_i^{\top}x) .
\end{eqnarray*}

\end{proof}

By the bi-conjugate theorem, we have $\phi(b)=\left( \phi^*(b) \right)^*$. Therefore, the Fenchel dual of Model~\eqref{BlockOpt} is
\begin{equation} \label{BO-Dual}
\phi(b)=
\begin{array}{ll}
\max_x & \left\{ \left\langle x,b\right\rangle-\sum_{i=1}^{m} f_i^*(A_i^{\top}x) \right\} \\
\end{array}
\end{equation}
which is in the form of finite-sum optimization. 
Similar to the discussions in the previous section, we have the following relationship between the optimal solution of \eqref{BlockOpt} and that of the dual problem \eqref{BO-Dual}.
\begin{lemma}
Assume that $x^\star$ is the optimal solution for the dual problem, and that $(y_1^\star, y_2^\star,\cdots,y_m^\star)$ is the optimal solution for the primal problem. Then,
\[
y_i^\star=\nabla f_i^*(A_i^{\top}x^\star), i=1,2,..., m.
\]
\end{lemma}

\begin{proof}
Since $x^\star$ is the optimal solution for the dual problem, by the first order optimality condition,
\[
b-\sum_{i=1}^{m}A_i\nabla f_i^*(A_i^{\top}x^\star)=0.
\]
Also, by the conjugacy relationship,
\[
\nabla f_i^*(A_i^{\top}x^\star)=\arg\min_{y_i} \left\{ f_i(y_i)-(x^{\star})^\top A_iy_i \right\}
\]
and by the optimality of $y_i^\star$ leads to
\[
y_i^\star=\nabla f_i^*(A_i^{\top}x^\star),\,\, i=1,2,..., m.
\]
\end{proof}

%

Since continuously differentiable functions are naturally Lipschitz continuous over any given compact sets, we observe the following:
\begin{lemma}
If $x$ is an $\epsilon$-optimal solution for the dual problem, then
$y_i=\nabla f_i^*(A_i^{\top}x)$ with $i=1,2,..., m$, is $O(\epsilon)$-optimal solution for the primal problem, satisfying
\begin{eqnarray*}
\|y_i-y_i^\star\| &\leq& \frac{\|A_i\|_2}{\mu_i}\|x-x^\star\| \\
\|A_{1} y_{1}+A_{2} y_{2}+\cdots+A_{m} y_{m}-b\| &\leq& \sum_{i=1}^{m}\frac{\|A_i\|_2^2}{\mu_i}\|x-x^\star\|,
\end{eqnarray*}
where $\|A_i\|_2$ is the spectral norm of $A_i$.
\end{lemma}

We know that $f_i^*$ is $1/L_i$-strongly convex and $1/\mu_i$-smooth. So $f_i^*(A_i^{\top}x)$ is $\lambda_{\min}(A_iA_i^{\top})/L_i$-strongly convex and $\|A_j\|_2^2/\mu_i$-smooth, with respect to $x$, where $\lambda_{\min}(A_iA_i^{\top})$ is the smallest eigenvalue value of $A_iA_i^{\top}$. Now, applying Theorem~\ref{thmssnm} on Model~\eqref{BO-Dual}, the following result follows readily:

\begin{theorem} \label{SSNM-BO-dual}
If we apply Generalized SSNM (Algorithm~\ref{Algorithm1}) to solve Model~\eqref{BO-Dual}, then by Theorem~\ref{thmssnm}
the gradient-oracle complexity to reach an $\epsilon$-optimal solution is
\[
O\left(\left(m+\frac{\sum_{j=1}^m\sqrt{1/\mu_j}\|A_j\|_2}{\sqrt{\sum_{i=1}^{m} \lambda_{\min}(A_iA_i^{\top})/L_i}}\right)\log(1/\epsilon)\right) ,
\]
where an $\epsilon$-optimal solution $y$ is such that $\|y-y^\star\|<\epsilon$.
\end{theorem}

Now, let us consider the constrained block optimization Model~\eqref{BlockOpt} without strong convexity assumption; that is,
\begin{equation} \label{BlockOpt-w}
\begin{array}{ll}
\min & f_{1}\left(y_{1}\right)+f_{2}\left(y_{2}\right)+\cdots+f_{m}\left(y_{m}\right) \\
\text { s.t. } & A_{1} y_{1}+A_{2} y_{2}+\cdots+A_{m} y_{m}=b
\end{array}
\end{equation}
where $f_i$ is $L_i$-smooth but not necessarily strongly convex. However, we assume that at least one of the optimal solutions
of \eqref{BlockOpt-w} is contained in a $D$-ball; i.e.\
$\|y_i^\star\|\leq D$, $i=1,2, ... , m$. Then, it is possible to introduce a positive perturbation to induce strong convexity. That is, we add a quadratic term $\frac\delta 2\|y_{i}\|^2$ to each function, where $\epsilon>0$ is the required precision and $\delta=\epsilon/mD^2$, to assure strong convexity:
\begin{equation} \label{BlockOpt-w-p}
\begin{array}{ll}
\min & f_{1}\left(y_{1}\right)+\frac\delta 2\|y_{1}\|^2+f_{2}\left(y_{2}\right)+\frac\delta 2\|y_{2}\|^2+\cdots+f_{m}\left(y_{m}\right)+\frac\delta 2\|y_{m}\|^2 \\
\text { s.t. } & A_{1} y_{1}+A_{2} y_{2}+\cdots+A_{m} y_{m}=b .
\end{array}
\end{equation}

%

Since in Model~\eqref{BlockOpt-w-p} each term $f_{i}\left(y_{i}\right)+\frac\delta 2\|y_{i}\|^2$ is $\delta$-strongly convex and $(L_i+\delta)$-smooth, we are in a position to apply
Theorem~\ref{SSNM-BO-dual} on Model~\eqref{BlockOpt-w-p} to yield the following iteration complexity assurance:
\begin{corollary}
If we apply Generalized SSNM (Algorithm~\ref{Algorithm1}) on Model~\eqref{BlockOpt-w-p}, then 
we can reach an $\epsilon$-solution in no more than
\[
O\left(\left(m+\frac{D\sqrt{m/\epsilon}\sum_{j=1}^m\|A_j\|_2}{\sqrt{\sum_{i=1}^{m} \lambda_{\min}(A_iA_i^{\top})/L_i}}\right)\log(1/\epsilon)\right)
\]
computations of $\nabla f_i$'s, where an $\epsilon$-optimal solution $y$ satisfies
\[
\left|\sum_{i=1}^m f_i(y_i)-\sum_{i=1}^m f_i(y_i^\star)\right|<\epsilon,\,\, \|A_{1} y_{1}+A_{2} y_{2}+\cdots+A_{m} y_{m}-b\|<\epsilon.
\]
\end{corollary}

\section{General Composite Optimization}  \label{sec6}

In this section, we consider a general composite optimization model:
\begin{equation}
\min _{x} F(x)=f(g_1(x),g_2(x),\cdots,g_m(x)), \label{com}
\end{equation}
where $g_i$ is $\mu_i$-strongly convex and $L_i$-smooth, and $f$ is convex and monotonically increasing with respect to each component, and in addition, in its domain we assume $0<b_i\leq \partial_i f\leq B_i$ for $i=1,2,...,m$. Finally, we assume $F$ is gradient-Lipschitz with $L$ being the Lipschitz gradient constant.

Since the finite sum is a special case of Model~\eqref{com}, the following result follows from Theorem~\ref{lower-iter-bound}.
\begin{theorem}
A lower bound on the iteration complexity for any first-order method to reach an $\epsilon$-solution of the above composite optimization problem \eqref{com} is
\begin{equation}
    \Omega 
    \left(\left(\frac{m}{\log m} + \max_{b_i\leq c_i\leq B_i} \left(\sum_{l=1}^{m}\sqrt{c_i L_i} \bigg/\sqrt{\sum_{i=1}^{m}c_i\mu_i}\right)\right)\log(1/\epsilon)\right) .
\end{equation}
\end{theorem}


Observe that
\begin{eqnarray*}
    & & \lambda F(x)+ (1-\lambda )F(y) \\
    &=& \lambda f(g_1(x),g_2(x),\cdots,g_m(x))+ (1-\lambda )f(g_1(y),g_2(y),\cdots,g_m(y)) \\
    &\geq& f(\lambda g_1(x) +(1-\lambda)g_1(y),\lambda g_2(x) +(1-\lambda)g_2(y),\cdots, \lambda g_m(x) +(1-\lambda)g_m(y)) \\
    &\geq& f\left( g_1(\lambda x+(1-\lambda)y)+\frac{\lambda(1-\lambda)\mu_1}{2}\|x-y\|^2,\cdots, g_m(\lambda x+(1-\lambda)y)+\frac{\lambda(1-\lambda)\mu_m}{2}\|x-y\|^2\right) \\
    &\geq& F(\lambda x+(1-\lambda)y)+\frac{\lambda(1-\lambda)\sum_{i=1}^mb_i\mu_i}{2}\|x-y\|^2.
\end{eqnarray*}
Therefore, $F$ is $\left(\sum_{i=1}^mb_i\mu_i\right)$-strongly convex. Similar to the finite sum case, we split the strongly convex part of the function. Define $h(x)=\frac{\sum_{i=1}^mb_i\mu_i}{2}\|x\|^2$, $\hat{F}(x)=F(x)-h(x)$. So, $F(x)=\hat{F}(x)+h(x)$, where $\hat{F}(x)$ is convex.

\subsection{A Generalized Katyusha Algorithm}

The gradient of the $F(x)$ can be calculated as
\[
\nabla F(x)=\sum_{i=1}^m \partial_i f(g_1(x),g_2(x),\cdots, g_m(x))\nabla g_i(x),
\]
which has the form of finite sum.
As we discussed earlier, Katyusha and SSNM are two efficient algorithms to solve the finite sum minimization model. In what follows, we shall generalize the Katyusha algorithm by adaptively changing the gradient estimator.

\begin{algorithm}[h]
\caption{Generalized Katyusha}
\label{Algorithm2}
\begin{algorithmic}
\STATE{$\sigma=\sum_{j=1}^mb_j\mu_j$}
\STATE{$L'=\max\{L,\sum_{i=1}^mB_{i}L_i\}$}
\STATE{$\tau_{2} = \frac{1}{2}, \tau_{1} = \min \left\{\frac{\sqrt{2m \sigma}}{\sqrt{3 L'}}, \frac{1}{2}\right\}, \alpha = \frac{1}{3 \tau_{1} L'}$}
\STATE{$y^{0}=z^{0}=\widetilde{x}^{0} \leftarrow x^{0}$}
\FOR{$s=0,1,..., S-1$}
\STATE{Calculate and memorize each $\nabla \hat{g}_j(\widetilde{x}^s)$}
\FOR{$j=0,1,..., 2m-1$}
\STATE{$k=2ms+j$}
\STATE{$x^{k+1} = \tau_{1} z^{k}+\tau_{2} \tilde{x}^{s}+\left(1-\tau_{1}-\tau_{2}\right) y^{k}$}
\STATE{Compute
\begin{eqnarray}
    \widetilde{\nabla}^{k+1} &=& \sum_{j=1}^m\partial_j f(g_1(x^{k+1}),g_2(x^{k+1}),\cdots, g_m(x^{k+1}))(\nabla \hat{g}_j(\widetilde{x}^s)+\mu_i x^{k+1}) \nonumber \\
    & & +\frac{1}{\pi_i}\partial_i f(g_1(x^{k+1}),g_2(x^{k+1}),\cdots, g_m(x^{k+1}))(\nabla \hat{g}_i(x^{k+1})-\nabla \hat{g}_i(\widetilde{x}^s))- \left(\sum_{j=1}^mb_j\mu_j\right)x^{k+1};\nonumber
\end{eqnarray}}
\STATE{$z^{k+1}=\arg \min _{z}\left\{\frac{1}{2 \alpha}\left\|z-z^{k}\right\|^{2}+\left\langle\widetilde{\nabla}^{k+1}, z\right\rangle+h(z)\right\}$}
\STATE{$y^{k+1} = \arg \min _{y}\left\{\frac{3 L'}{2}\left\|y-x^{k+1}\right\|^{2}+\left\langle\widetilde{\nabla}^{k+1}, y\right\rangle+h(y)\right\}$}

\ENDFOR
\STATE{$\tilde{x}^{s+1} =\left(\sum_{j=0}^{2m-1}(1+\alpha \sigma)^{j}\right)^{-1} \cdot\left(\sum_{j=0}^{2m-1}(1+\alpha \sigma)^{j} \cdot y^{2ms +j+1}\right)$}
\ENDFOR
\STATE \textbf{return} $\widetilde{x}^S$
\end{algorithmic}
\end{algorithm}



Algorithm Katyusha is based on the notion of the Katyusha momentum while equipping with an SVRG gradient estimator. Specifically, it sets the gradient estimator $\widetilde{\nabla}^{k}$ to be the SVRG gradient estimator. In the generalized Katyusha algorithm (Algorithm~\ref{Algorithm2}), we shall use the following gradient estimator for the general composite problem. To simplify the expressions, define $\hat{g}_i(x)=g_i(x)-\frac{\mu_i}{2}\|x\|^2$, which is convex and $L_i$-smooth. Let
\begin{eqnarray} \label{SGD-est}
    \widetilde{\nabla}^{k+1}
    &=& \sum_{j=1}^m\partial_j f(g_1(x^{k+1}),g_2(x^{k+1}),\cdots, g_m(x^{k+1}))(\nabla \hat{g}_j(\widetilde{x})+\mu_i x^{k+1}) \nonumber \\
    & & +\frac{1}{\pi_i}\partial_i f(g_1(x^{k+1}),g_2(x^{k+1}),\cdots, g_m(x^{k+1}))(\nabla \hat{g}_i(x^{k+1})-\nabla \hat{g}_i(\widetilde{x}))- \left(\sum_{j=1}^mb_j\mu_j\right)x^{k+1},
    \nonumber
    \end{eqnarray}
where $i\in [m]$ is sampled with probability $\pi_i=\frac{B_i L_i}{\sum_{j=1}^m B_j L_j}$, $i\in [m]$. Let us denote the discrete distribution over $[m]$ with the previous probabilities as $\mathcal{D}$. It is easy to verify that $\widetilde{\nabla}^{k+1}$ is an unbiased estimation of $\nabla \hat{F}(x^{k+1})$.
Let $L$ be the Lipschitz gradient constant of the objective function $F$. Below is an upper bound on the iteration complexity of Algorithm~\ref{Algorithm2}.
\begin{theorem} \label{CompOpt}
If we apply Algorithm~\ref{Algorithm2} on Model~\eqref{com}, then in
\[
K = S \cdot m
=O\left(\left(m+ \sqrt{\frac{mL+m \sum_{i=1}^mB_{i}L_i}{ \sum_{i=1}^mb_i\mu_i}} \right) \cdot \log \frac{F\left(x^{0}\right)-F\left(x^{\star}\right)}{\epsilon}\right)
\]
number of iterations, we shall arrive at a solution $x^K$ satisfying
$\mathbb{E}\left[F(x^K)\right]-F(x^{\star})\leq \epsilon$. Moreover, at each iteration we need to compute one $\nabla g_i$, the values of $g_i$, and $\partial_i f$ for $i=1,2,...,m$.
\end{theorem}

Our analysis follows the same line as in the analysis of the original Katyusha~\cite{allen2017katyusha}. First, we note an upper bound on the variance of the unbiased gradient estimator, which is the following lemma:
\begin{lemma}
\begin{eqnarray*}
\mathbb{E}\left[\left\|\widetilde{\nabla}^{k+1}-\nabla \hat{F}\left(x^{k+1}\right)\right\|^{2}\right] \leq  2\left(\sum_{i=1}^mB_{i}L_i\right)\left(\hat{F}(\widetilde{x})- \hat{F}(x^{k+1})-\left\langle\nabla \hat{F}(x^{k+1}), \widetilde{x}-x^{k+1}\right\rangle\right).
\end{eqnarray*}
\end{lemma}

\begin{proof}
Recall that $\hat{g}_i(x)=g_i(x)-\frac{\mu_i}{2}\|x\|^2$ is convex and $L_i$-smooth, and Theorem 2.1.5 of Nesterov's textbook \cite{nesterov2018lectures} suggests that
\begin{equation} \label{nesinq}
\|\nabla \hat{g}_i(x^{k+1})-\nabla \hat{g}_i(\widetilde{x})\|^2\leq 2 L_i \left(\hat{g}_i(\widetilde{x})- \hat{g}_i(x^{k+1})-\left\langle\nabla \hat{g}_i(x^{k+1}), \widetilde{x}-x^{k+1}\right\rangle\right).
\end{equation}

Also note,
\begin{equation}  \label{rv-ineq}
\mathbb{E}\|\zeta-\mathbb{E} \zeta\|^{2}=\mathbb{E}\|\zeta\|^{2}-\|\mathbb{E} \zeta\|^{2}\leq \mathbb{E}\|\zeta\|^{2}
\end{equation}
for any random variable $\zeta$, and by the convexity property,
\begin{equation} \label{conv-F}
    \sum_{i=1}^m\partial_i f(g_1(x^{k+1}),g_2(x^{k+1}),\cdots, g_m(x^{k+1}))\cdot \left(g_i(\widetilde{x})- g_i(x^{k+1})\right)\leq F(\widetilde{x})- F(x^{k+1}).
\end{equation}

Therefore, we have 
\begin{eqnarray*}
    & & \mathbb{E}\left[\left\|\widetilde{\nabla}^{k+1}-\nabla \hat{F}\left(x^{k+1}\right)\right\|^{2}\right]\\
    &=& \mathbb{E}_{ i \sim \mathcal{D}}\left[\left\|\sum_{j=1}^m\partial_j f(g_1(x^{k+1}),g_2(x^{k+1}),\cdots, g_m(x^{k+1}))(\nabla \hat{g}_j(\widetilde{x})-\nabla \hat{g}_j(x^{k+1}))\right.\right.\\
    &  & \left.\left.+\frac{1}{\pi_i}\partial_i f(g_1(x^{k+1}),g_2(x^{k+1}),\cdots, g_m(x^{k+1}))(\nabla \hat{g}_i(x^{k+1})-\nabla \hat{g}_i(\widetilde{x}))\right\|^{2}\right]\\
    &\overset{\eqref{rv-ineq}}{\le} & \mathbb{E}_{i \sim \mathcal{D}}\left[\left\|\frac{1}{\pi_i}\partial_i f(g_1(x^{k+1}),g_2(x^{k+1}),\cdots, g_m(x^{k+1}))(\nabla \hat{g}_i(x^{k+1})-\nabla \hat{g}_i(\widetilde{x}))\right\|^{2}\right]\\
    &\overset{\eqref{nesinq}}{\leq} & \sum_{i=1}^m \frac{2 B_{i}L_i}{\pi_i}\partial_i f(g_1(x^{k+1}),g_2(x^{k+1}),\cdots, g_m(x^{k+1}))
    \left(\hat{g}_i(\widetilde{x})- \hat{g}_i(x^{k+1})-\left\langle\nabla \hat{g}_i(x^{k+1}), \widetilde{x}-x^{k+1}\right\rangle\right)\\
    &=&\sum_{i=1}^m \frac{2 B_{i}L_i}{\pi_i}\partial_i f(g_1(x^{k+1}),g_2(x^{k+1}),\cdots, g_m(x^{k+1}))\cdot\\
    & & \left(g_i(\widetilde{x})- g_i(x^{k+1})-\left\langle\nabla g_i(x^{k+1}), \widetilde{x}-x^{k+1}\right\rangle-\frac{\mu_i}{2}\|\widetilde{x}-x^{k+1}\|^2\right)\\
    & \overset{\eqref{conv-F}}{\leq} & 2\left(\sum_{i=1}^mB_{i}L_i\right)\left(F(\widetilde{x})- F(x^{k+1})-
    \left\langle\nabla F(x^{k+1}), \widetilde{x}-x^{k+1}\right\rangle-\frac{\sum_{i=1}^mb_i\mu_i}{2}\|\widetilde{x}-x^{k+1}\|^2\right)\\
    &=&2\left(\sum_{i=1}^mB_{i}L_i\right)\left(\hat{F}(\widetilde{x})- \hat{F}(x^{k+1})-\left\langle\nabla \hat{F}(x^{k+1}), \widetilde{x}-x^{k+1}\right\rangle\right),
\end{eqnarray*}
which is the desired result.
\end{proof}

Using the above lemma, the rest of the analysis is identical to that in \cite{allen2017katyusha}, and we omit the details here for succinctness.

\subsection{Implementation Under an Additional Assumption}

In Algorithm~\ref{Algorithm2}, computing the unbiased gradient estimator requires the information of all $\partial_i f$, which can be relaxed if a somewhat stronger assumption on the partial derivative holds true (see Assumption~\ref{strass} below).

Let $\nabla_i F(x)=\partial_i f(g_1(x),g_2(x),\cdots, g_m(x))\nabla g_i(x)$. Then $\nabla F(x)=\sum_{i=1}^m\nabla_i F(x)$. Define
\[
\hat{\nabla}_i F(x)=\nabla_i F(x)-b_i\mu_i x.
\]
We have $\nabla\hat{F}(x)=\nabla F(x)-\left(\sum_{i=1}^mb_i\mu_i\right) x=\sum_{i=1}^m\hat{\nabla}_i F(x)$.
\begin{assumption}\label{strass}
For any $1\le i \le m$, we assume that $\hat{\nabla}_i F(x)$ satisfies the following condition
\[
\left\|\hat{\nabla}_i F\left(y\right)-\hat{\nabla}_i F(x)\right\|^{2} \leq 2 l_i\int_0^{1}\left\langle \hat{\nabla}_i F\left(x+t(y-x)\right)-\hat{\nabla}_i F\left(x\right), y-x\right\rangle dt,
\]
where $l_i>0$ is a constant.
\end{assumption}

\begin{remark}
{\rm Theorem 2.1.5 of \cite{nesterov2018lectures} shows that Assumption~\ref{strass} holds when $\hat{\nabla}_i F(x)$ is the gradient of a convex function with gradient Lipschitz constant $l_i$.}
\end{remark}

We use the following way to give a stochastic gradient approximation of $\hat{F}(x^{k+1})$. Let distribution $\mathcal{D}$ be to output $i\in [n]$ with probability $\pi_i=\frac{l_i}{\sum_{j\in [n]} l_j}$, and introduce the gradient estimator as
\begin{equation} \label{new-grad-est}
\widetilde{\nabla}^{k+1} := \nabla \hat{F}(\widetilde{x})+\frac{1}{\pi_i}(\hat{\nabla}_i F(x^{k+1})-\hat{\nabla}_i F(\widetilde{x})).
\end{equation}

Unlike the estimator \eqref{SGD-est}, in \eqref{new-grad-est} 
we only need to compute two $\partial_i f$'s during each iteration. Moreover, we no longer need to store all of $\nabla g_i(\widetilde{x})$.

\begin{theorem}
Under Assumption~\ref{strass} if we use the stochastic gradient estimator \eqref{new-grad-est} in Algorithm~\ref{Algorithm2}, then in
\[
K := S \cdot m=O\left(\left(m+\sqrt{ \frac{m \sum_{i=1}^ml_i }{\sum_{i=1}^mb_i\mu_i}}\right) \cdot \log \frac{F\left(x^{0}\right)-F\left(x^{\star}\right)}{\epsilon}\right)
\]
iterations we will have
\[
\mathbb{E}\left[F(x^K)\right]-F(x^{\star})\leq \epsilon .
\]
In that case, at each iteration we need to compute one $\nabla g_i$, the values of all $g_i$ and one $\partial_i f$.
\end{theorem}

Similar to the analysis before, we only need to provide an upper bound of the MSE of the gradient approximation.
\begin{lemma}
It holds that
\[ \mathbb{E}\left[\left\|\widetilde{\nabla}^{k+1}-\nabla \hat{F}\left(x^{k+1}\right)\right\|^{2}\right] \\
\leq 2\left(\sum_{i=1}^m l_i\right) \left(\hat{F}(\widetilde{x})- \hat{F}(x^{k+1})-\left\langle\nabla \hat{F}(x^{k+1}), \widetilde{x}-x^{k+1}\right\rangle\right) .\]
\end{lemma}

\begin{proof}
\begin{eqnarray*}
    & & \mathbb{E}\left[\left\|\widetilde{\nabla}^{k+1}-\nabla \hat{F}\left(x^{k+1}\right)\right\|^{2}\right]\\
    &=&\mathbb{E}_{i \sim \mathcal{D}}\left[\left\|\nabla \hat{F}(\widetilde{x})+\frac{1}{\pi_i}(\hat{\nabla}_i F(x^{k+1})-\hat{\nabla}_i F(\widetilde{x}))-\nabla \hat{F}\left(x^{k+1}\right)\right\|^{2}\right]\\
    &=&\mathbb{E}_{i \sim \mathcal{D}}\left[\left\|\frac{1}{\pi_i}(\hat{\nabla}_i F(x^{k+1})-\hat{\nabla}_i F(\widetilde{x}))+\left(\nabla \hat{F}(\widetilde{x})-\nabla \hat{F}\left(x^{k+1}\right)\right)\right\|^{2}\right]\\
    &\overset{\eqref{rv-ineq}}{\leq} & \mathbb{E}_{i \sim \mathcal{D}}\left[\left\|\frac{1}{\pi_i}(\hat{\nabla}_i F(x^{k+1})-\hat{\nabla}_i F(\widetilde{x}))\right\|^{2}\right]\\
    &\overset{{\tiny \rm Assumption~\ref{strass}}}{\leq} & \sum_{i \in[n]} \frac{2 l_{i}}{\pi_i}\int_0^{1}\left\langle\hat{\nabla}_i F(x^{k+1}+t(\widetilde{x}-x^{k+1}))-\hat{\nabla}_i F(x^{k+1}), \widetilde{x}-x^{k+1}\right\rangle dt\\
    &=& 2\left(\sum_{i=1}^ml_i\right) \int_0^{1}\sum_{i=1}^m\left\langle\hat{\nabla}_i F(x^{k+1}+t(\widetilde{x}-x^{k+1}))-\hat{\nabla}_i F(x^{k+1}), \widetilde{x}-x^{k+1}\right\rangle dt\\
    &=& 2\left(\sum_{i=1}^m l_i\right) \int_0^{1}\left\langle\nabla \hat{F}(x^{k+1}+t(\widetilde{x}-x^{k+1}))-\nabla \hat{F}(x^{k+1}), \widetilde{x}-x^{k+1}\right\rangle dt\\
    &=& 2\left(\sum_{i=1}^m l_i\right) \left(\hat{F}(\widetilde{x})- \hat{F}(x^{k+1})-\left\langle\nabla \hat{F}(x^{k+1}), \widetilde{x}-x^{k+1}\right\rangle\right) .
\end{eqnarray*}

\end{proof}

\section{Numerical Experiments} \label{sec7}

In this section, we shall conduct some experiments to test the efficacy of the algorithms being studied in this paper.
The organization is as follows. In Subsection \ref{exp1} we consider two examples to test the performance of generalized SSNM algorithm amid some commonly used algorithms that do not incorporate different Lipschitz gradient parameters. In Subsection \ref{exp3} we test the performance of applying the generalized SSNM algorithm on solving the dual of a constrained strongly convex multi-block optimization model. In Subsection \ref{exp4} we experiment with a general composite problem to test the generalized Katyusha algorithm compared with the standard (deterministic) accelerated gradient method of Nesterov.

\subsection{Weighted Least Squares and Weighted Logistic Regression} \label{exp1}

In this subsection, we apply Algorithm~\ref{Algorithm1} on two different types of finite sum problems,  to compare the generalized SSNM algorithm with other algorithms.

The first example is the weighted Least Square problem in $x\in \mathbb{R}^{n}$.
\[
\frac{1}{m} \sum_{i=1}^{m} w_i\left(a_{i}^{\top}x-b_{i}\right)^2+\frac{\mu}{2}\|x\|^{2}
\]
where $m=10000$, $n=100$, $\mu=10^{-5}$. To test the relevance of varying Lipschitz gradient parameters, we choose to weigh the gradient Lipschitz constants differently. Specifically, we introduce $w$'s consisting of $\sqrt{m}$-number elements $m$ and $(m-\sqrt{m})$ elements of $1$, and $A=(a_{1},a_{2},\cdots, a_{m})^{\top}$ is an $m \times n$ data matrix generated from the standard normal distribution. Then, we scale $A$ to ensure that the estimated gradient Lipschitz parameters of the first term $\frac{1}{m} \sum_{i=1}^{m} w_i\left(a_{i}^{\top}x-b_{i}\right)^2$ is $1$.

The second example is the weighted Logistic Regression problem.
\[
\frac{1}{m} \sum_{i=1}^{m} w_i\log \left(1+\exp \left(-b_{i} a_{i}^{\top} x\right)\right)+\frac{\mu}{2}\|x\|^{2}
\]
where $m=10000$, $n=100$, $\mu=10^{-5}$. We use the same approach to generate and scale matrix $A$.

We have tested the algorithms with the following specifics.
\begin{itemize}
    \item SAGA \cite{defazio2014saga}: we tune the learning rate only.
    \item SSNM \cite{zhou2019direct}: we tune the learning rate $\eta$ and the parameter $\tau$ only.
    \item SVRG \cite{johnson2013accelerating}: we tune the learning rate only.
    \item Katyusha \cite{allen2017katyusha}: we tune the learning rate only.
    \item Generalized SSNM (Algorithm~\ref{Algorithm1}): we tune the learning rate  $\eta$ and the parameter $\lambda$ only.
\end{itemize}

To give a fair comparison, for all the parameters we tune, we select learning rates from the set $\left\{10^{-k}, 3 \times 10^{-k}:k \in \mathbb{Z}\right\}$ times the parameter settings for theoretical analysis in their corresponding paper. The results are presented in Figure \ref{fig1} and Figure \ref{fig2}. Each effective pass counts calculating $\nabla f_i(x)$ for $m$ times. We can observe that the generalized SSNM (Algorithm~\ref{Algorithm1}) converges considerably faster than any other algorithms, once the information about different Lipschitz gradients is available. This behavior is not surprising, of course. It is, however, informative to know the degree to which more exact Lipschitz constants could add to convergence, in the context of stochastic gradient methods.

\begin{figure}
  \begin{minipage}[t]{0.5\linewidth}
    \centering
    \includegraphics[scale=0.2]{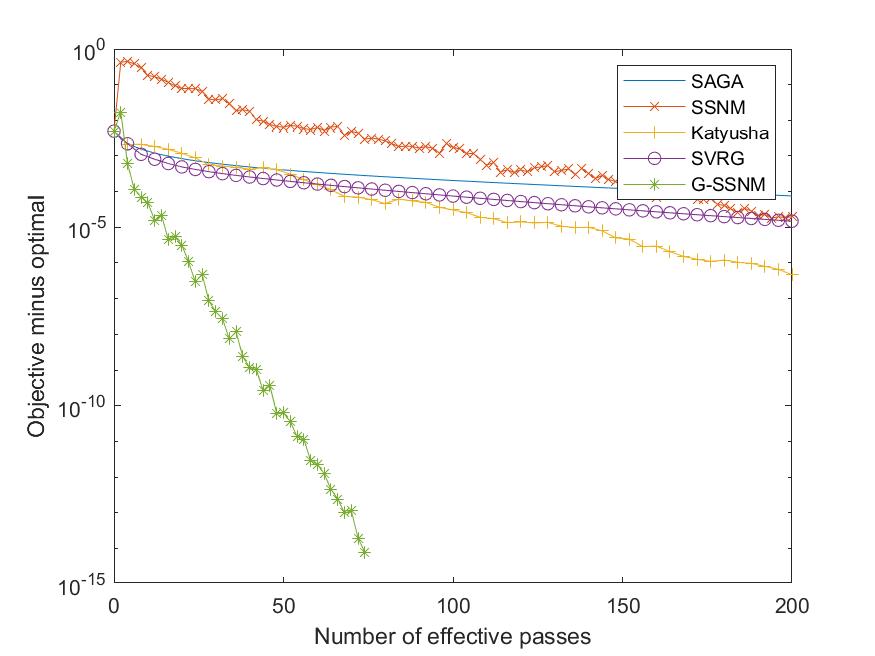}
    \caption{Weighted Least Squares}
    \label{fig1}
  \end{minipage}%
  \begin{minipage}[t]{0.5\linewidth}
    \centering
    \includegraphics[scale=0.2]{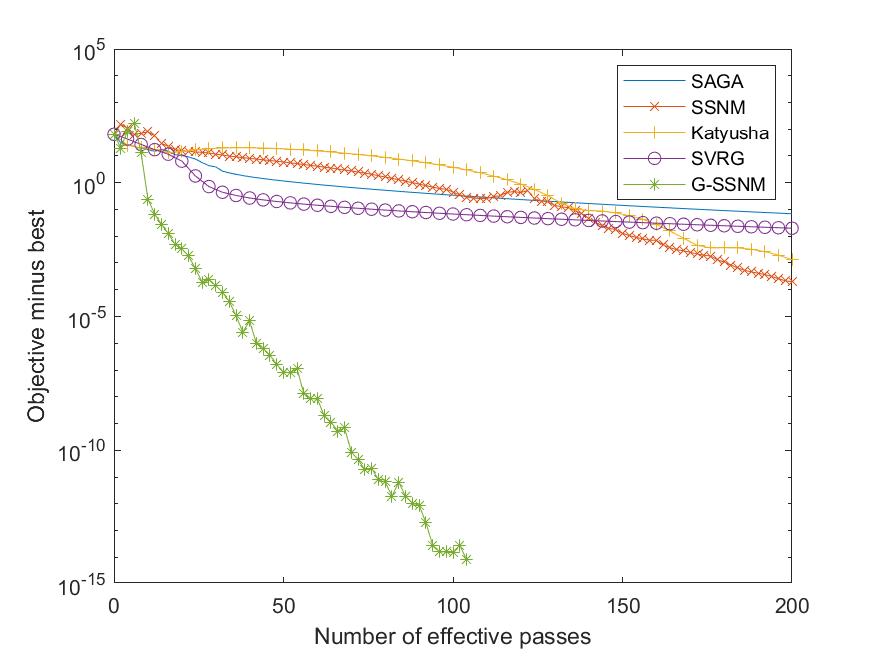}
    \caption{Weighted Logistic Regression}
    \label{fig2}
  \end{minipage}
\end{figure}

\subsection{Strongly Convex Multi-block Optimization} \label{exp3}
In this subsection, we consider the following constrained strongly convex multi-block problem optimization model:
\begin{equation}
\begin{array}{ll}
\min & \sum_{i=1}^{m} \frac 1 2 y_i^{\top}P_i y_i+ a_i^{\top} y_i \\
\text { s.t. } & y_{1}+y_{2}+\cdots+ y_{m}=m b \\
\end{array}
\end{equation}
where $m=10$, $n=10$. Each $P_i\in \mathbb{R}^{n\times n}$ is positive definite, $i=1,2,...,m$. We first generate its eigenvalues following uniform distribution in $[0,1]$ and adjust its minimum into $\mu=10^{-3}$. Then, we generate its orthogonal matrix randomly to get definite symmetrical matrix $P_i\in \mathbb{R}^{n\times n}$, $a_i\in \mathbb{R}^{n}$ and $b\in \mathbb{R}^{n}$ are generated following standard Gaussian distribution. We follow the same rate-tuning procedure as in the numerical experiments conducted in Subsection~\ref{exp1} to tune the parameters.

Figure \ref{fig3} shows the convergence behavior of our method, which is observed to converge to the global solution linearly. We also compare it with  Alternating Direction Method of Multipliers (ADMM) and Alternating Proximal Gradient Method of Multipliers (APGMM) \cite{gao2018low}. Bear in mind that those three algorithms require very different types of subroutines: ADMM requires solving each separated sub-problem exactly; APGMM only requires computing the gradient of each block of the Lagrangian; our new method only needs the gradient information of the conjugate function.

\begin{figure}
  \begin{minipage}[t]{0.5\linewidth}
    \centering
    \includegraphics[scale=0.2]{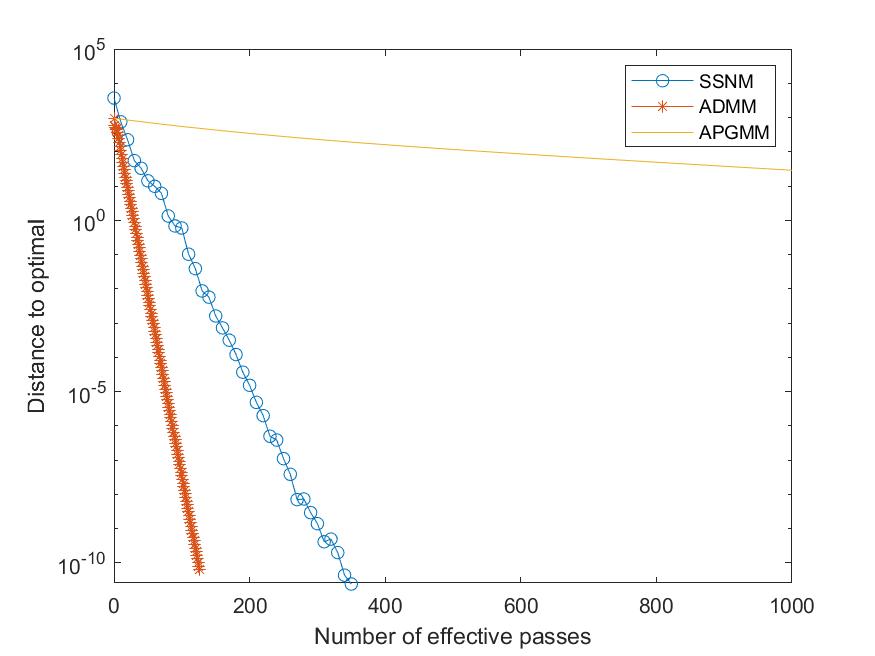}
    \caption{Strongly Convex Multi-block problem}
    \label{fig3}
  \end{minipage}%
  \begin{minipage}[t]{0.5\linewidth}
    \centering
    \includegraphics[scale=0.2]{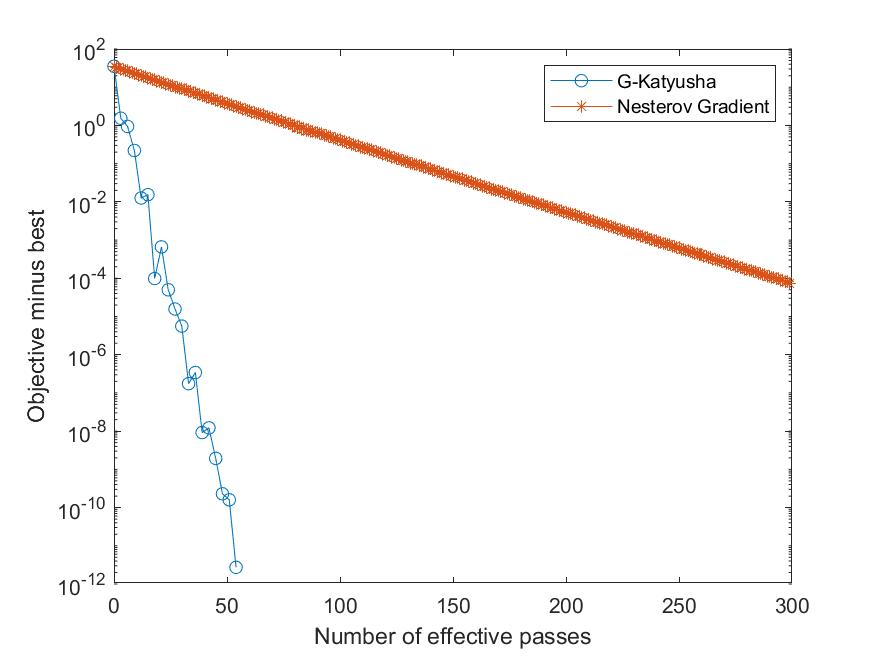}
    \caption{Composite problem}
    \label{fig4}
  \end{minipage}
\end{figure}

\subsection{General Composite Optimization} \label{exp4}

In this subsection, we shall experiment with Algorithm~\ref{Algorithm2} (Generalized Katyusha) with the following composite optimization problem $F(x)=f(g(x))$, where $f:\mathbb{R}^m\rightarrow  \mathbb{R}$ and $g:=(g_1,g_2,\cdots, g_m)^{\top}$, $g_i:\mathbb{R}^n\rightarrow  \mathbb{R}$, $i=1,2,\cdots, m$. In our experiments, we set $m=80$, $n=80$. We generate $f$ by
\[f(y)=y^\top Q y,\]
where $Q\succ 0$ is a randomly doubly positive matrix (that is, each component of $Q$ is positive).
We then generate each $g_i(x)$ as a random quadratic function
\[
g_i(x)=\frac 1 2 x^{\top}P_i x+ a_i^{\top} x+r_i,
\]
where $P_i$ is positive definite, $i=1,2,...,m$. We follow the same procedure as in Subsection~\ref{exp3} to generate $P_i$ with minimum eigenvalue $\mu=10^{-5}$, so that $g_i(x)$ is positive and $\mu$-strongly convex. In our experiments, all of the required parameters from the problem settings are known as the way they are generated. The results are presented in Figure~\ref{fig4}. Each effective pass means calculating $\nabla g_i(x)$ for $m$ times. The generalized Katyusha algorithm converges linearly and is faster than the deterministic Nesterov's accelerated gradient method, as measured by the total amount of gradient computations. As one may observe, Algorithm~\ref{Algorithm2} (Generalized Katyusha) manages to perform $\Theta(m)$ iterations in each effective pass while Nesterov's accelerated gradient method can only perform one iteration per pass.

\section{Conclusion} \label{sec8}

In this paper, we studied the lower and upper bounds on the gradient computation numbers of first order methods for solving the finite sum convex optimization problem, where the condition numbers of the functions are different. For the vanilla finite-sum strongly convex optimization (with varying condition numbers) model, we presented an accelerated stochastic gradient algorithm based on variance reduction, and showed that the algorithm essentially matches the information lower bound that we obtained earlier in the paper, up to a logarithm factor. This approach gives rise to a dual method to solve constrained block-variable optimization via Fenchel duality, yielding a better gradient computation bound than that of its primal counter-part.

Following up on the finite-sum model, we went on to consider a general composite convex optimization model, and proposed a generalized Katyusha algorithm. In this case, though there is still a gap between the upper bound and lower bound in terms of the gradient computation counts, the generalized Katyusha algorithm converges faster than the vanilla accelerated gradient method in terms of the amount of gradient computations required. How to narrow down the theoretical gap between the lower and upper bounds for the general model remains an interesting question for the future research.

As a related model, we also provided a lower bound on the gradient computations for constrained separate block-variable convex optimization. The dual approach introduced earlier provides an upper bound on the gradient computation counts. The upper and lower bounds meet when all the functions share the same condition numbers. In the general case however, a gap exists. It will be interesting to further improve either the upper or the lower bounds. We believe that those questions are fundamental in understanding the information structure for the first-order methods applied to solve composite convex optimization.








\bibliographystyle{siamplain}
\bibliography{sample}
\end{document}